\documentclass{article}%
\usepackage{amsfonts}
\usepackage{amsmath}
\usepackage{amssymb}
\usepackage{graphicx}%
\RequirePackage[a4paper,top=2.54cm,bottom=2.54cm,left=1.90cm,right=1.90cm,%
                headsep=1em,includehead,includefoot]{geometry}
\RequirePackage{float}
\RequirePackage{indentfirst}
\providecommand{\U}[1]{\protect\rule{.1in}{.1in}}
\newtheorem{theorem}{Theorem}

\newtheorem{corollary}[theorem]{Corollary}

\newtheorem{definition}[theorem]{Definition}

\newtheorem{Lemma}[theorem]{Lemma}

\newtheorem{proposition}[theorem]{Proposition}
\newtheorem{remark}[theorem]{Remark}

\newtheorem{hypothesis}[theorem]{Hypothesis}

\newenvironment{proof}[1][Proof]{\noindent\textbf{#1.} }{\ \rule{0.5em}{0.5em}}

\providecommand{\keywords}[1]
{
  \small	
  \textbf{\textit{Keywords:}} #1
}

\begin{document}
\title{Inviscid Limit for Stochastic Second-Grade Fluid Equations}
\date{}
\author{Eliseo Luongo$^{1}$ \\ $^1$ Scuola Normale Superiore, Piazza dei Cavalieri,7, 56126 Pisa, Italy \\ Email: eliseo.luongo@sns.it}
\maketitle

\begin{abstract}
  We consider in a smooth and bounded two dimensional domain the convergence in the $L^2$ norm, uniformly in time, of the solution of the stochastic second-grade fluid equations with transport noise and no-slip boundary conditions to the solution of the corresponding Euler equations. We prove, that assuming proper regularity of the initial conditions of the Euler equations and a proper behavior of the parameters $\nu$ and $\alpha$, then the inviscid limit holds without requiring a particular dissipation of the energy of the solutions of the second-grade fluid equations in the boundary layer.
\end{abstract}

\maketitle

\section{Introduction}\label{intro}
The second-grade fluid equations are a model for viscoelastic fluids, with two parameters: $\alpha>0$, corresponding to the elastic response, and $\nu>0$, corresponding to viscosity. Considering a constant density, $\rho=1$, their stress tensor is given by 
\begin{equation*}
    T=-pI+\nu A_1+\alpha^2 A_2-\alpha^2 A_1^2,
\end{equation*}
where \begin{align*}
    A_1=\frac{\nabla u+\nabla u^T}{2},\\
    A_2=\partial_t A_1+A_1\nabla u+\nabla u^T A_1,
\end{align*}
being $p$ the pressure and $u$ the velocity field.
Given this stress tensor, the equations of motion for an incompressible homogeneous fluid of grade 2 are given by
\begin{align}\label{second-grade system}
\begin{cases}
\partial_t v=\nu \Delta u-\operatorname{curl}(v)\times u+\nabla p+f\\
\operatorname{div}u=0\\
v=u-\alpha^2\Delta u\\
u|_{\partial D}=0\\ 
u(0)=u_0.
\end{cases}    
\end{align}
where $f$ describes some external forces, possibly stochastic, acting on the fluid, see \cite{dunn1974thermodynamics}, \cite{rivlin1997stress} for further details on the physics behind this system.
The analysis of the deterministic system started with \cite{cioranescu1984existence}. They proved global existence and uniqueness without restricting the problem to the two dimensional case. Setting, formally, $\alpha=0$ in equations \eqref{second-grade system} we can reduce the system to the well-known Navier-Stokes one:
\begin{align}\label{Navier-Stokes system}
\begin{cases}
\partial_t u=\nu \Delta u-u\cdot\nabla u+\nabla p+f\\
\operatorname{div}u=0\\
u|_{\partial D}=0\\ 
u(0)=u_0.
\end{cases}    
\end{align}
Thus \eqref{second-grade system} can be seen as a generalization of \eqref{Navier-Stokes system}. Moreover, in \cite{iftimie2002remarques}, it has been shown that second-grade fluid equations are a good approximation of the Navier-Stokes system. Due to these good properties of the system it is a legitimate question trying to understand if the second-grade fluid equations behave better than the Navier-Stokes ones in problems related to turbulence, like the inviscid limit for domain with boundary and no-slip boundary conditions. In fact, such question is far for being solved for system \eqref{Navier-Stokes system} also in the deterministic framework. Partial results are available: \begin{enumerate}
    \item Unconditioned results. They are based on strong assumptions about the flows. For example flows with radial symmetry \cite{lopes2008vanishing}, \cite{lopes2008vanishing2}, or flows with analytic boundary layers \cite{maekawa2014inviscid}, \cite{sammartino1998zero}.
    \item Conditioned results. They are based on stating some criteria about the behavior of the solutions of the Navier-Stokes equations in the boundary layer in order to prove the inviscid limit. This line of research started with the famous work by Kato \cite{kato1984remarks}, see \cite{constantin2015inviscid}, \cite{temam1997behavior}, \cite{wang2001kato} for other results. For what concern the Stochastic framework few results are available, see for example \cite{luongo2021inviscid} for a generalization of the Kato's results to the additive noise case and a wilder set of initial conditions. 
\end{enumerate}
The analysis of the inviscid limit for the deterministic second-grade fluid equations is a partially well-understood topic. In particular, in \cite{lopes2015approximation}, the authors show that the behavior of the system changes considering different scaling between $\nu$ and $\alpha^2$.

If we set, formally, $\nu=0$ in system \eqref{second-grade system} second-grade fluid equations reduce to the so-called Euler-$\alpha$ equations:
\begin{align}\label{euler alpha system}
\begin{cases}
\partial_t v=-\operatorname{curl}(v)\times u+\nabla p+f\\
\operatorname{div}u=0\\
v=u-\alpha^2\Delta u\\
u|_{\partial D}=0\\ 
u(0)=u_0.
\end{cases}    
\end{align}
This system models the averaged motion of an ideal incompressible fluid when filtering over spatial scales smaller than $\alpha$ and its well-posedness has been treated in \cite{marsden2000geometry},\cite{shkoller2000analysis}. 
Euler-$\alpha$ equations, formally, satisfies the condition of \cite[Theorem 3]{lopes2015approximation}. Therefore we can expect that the inviscid limit holds also in this framework. Indeed, this is true as has been showed in \cite{lopes2015convergence}. 

In this work, we will consider stochastic second-grade fluid equations and stochastic Euler-$\alpha$ equations with transport noise which scales with respect to the elasticity. We want to understand if the good behavior proved in \cite{lopes2015approximation} if $\nu=O(\alpha^2)$ and in \cite{lopes2015convergence} if $\nu=0$ is preserved also in this case. There are several motivations to consider transport noise, as the effect of small scales on large scales in fluid dynamics problems, see \cite{debussche2022second}, \cite{flaWas}, \cite{flandoli2022additive},  \cite{holm2015variational} for several discussions on this topic. A first issue related to the analysis of the inviscid limit in the case of the transport noise is the well-posedness of the systems. In fact the existence of strong probabilistic solution of such systems is outside the framework treated in \cite{razafimandimby2017grade} and \cite{razafimandimby2012strong}, thus we need to improve slightly these results thanks to the properties of the transport noise. In the following $\nu\geq 0$ and we will always speak of second-grade fluid equations even if $\nu=0.$

The paper is organized as follows. In Section \ref{results} we introduce the mathematical problem, we state our main theorems and we give some well-known results for the the Euler equations and the analysis of the stochastic second-grade fluid equations. In
Sections \ref{well posed} we prove that the stochastic second-grade fluid equations with transport noise and no-slip boundary conditions are well posed. In Section \ref{energy estimates}, thanks to the already proven well-posedness and Hypothesis \ref{Hypothesis inviscid limit} below we improve the energy estimates obtained in Section \ref{well posed} in order to get some estimates crucial for the proof of Theorem \ref{Main thm}. The proof of our main theorem on the inviscid limit occupies Section \ref{proof of main thm}. Lastly in Section \ref{additive noise} we add some remarks for the analysis of the additive noise case.
\section{Main Results}\label{results}
Let us start this section introducing some general assumptions which will be always adopted under our analysis even if not recalled.
\begin{hypothesis}\label{General Hypothesis}
\begin{itemize}
    \item $0<T<+\infty$.
    \item $D$ is a bounded, smooth, simply connected domain.
    \item $\left(\Omega,\mathcal{F},\mathcal{F}_t,\mathbb{P}\right)$ is a filtered probability space such that $(\Omega, \mathcal{F},\mathbb{P})$ is a complete probability space, $(\mathcal{F}_t)_{t\in [0,T]}$ is a right continuos filtration and $\mathcal{F}_0$ contains every $\mathbb{P}$ null subset of $\Omega$.
\end{itemize}
\end{hypothesis}
For square integrable semimartingales taking value in separable Hilbert spaces $U_1,\ U_2$ we will denote by
$[M, N]_t$ the quadratic covariation process. If $M, N$ take values in the same separable Hilbert space $U$ with orthonormal basis $u_i$, we will denote by $\langle\langle M,N\rangle\rangle_t=\sum_{i\in \mathbb N} [\langle M,u_i\rangle_U, \langle N,u_i\rangle_U]_t$.
For each $k\in \mathbb{N},\ 1\leq p\leq \infty$ we will denote by $L^p(D)$ and $W^{k,p}(D)$ the well-known Lebesgue and Sobolev spaces. We will denote by $C_{c}^{\infty}(D)$ the space of smooth functions with compact support and by $W^{k,p}_0(D)$ their closure with respect to the $W^{k,p}(D)$ topology.  If $p=2$, we will write $H^k(D)$ (resp. $H^k_0(D)$ ) instead of $W^{k,2}(D)$ (resp. $W^{k,2}_0(D)$). Let $X$ be a separable Hilbert space, denote by $L^p(\mathcal{F}_{t_0},X)$ the space of $p$ integrable random variables with values in $X$, measurable with respect to $\mathcal{F}_{t_0}$. We will denote by $L^p(0,T;X)$ the space of measurable functions from $[0,T]$ to $X$ such that $$\lVert u\rVert_{L^p(0,T;X)}:=\left(\int_0^T \lVert u(t) \rVert_X^p\ dt\right)^{1/p}<+\infty,\ \ 1\leq p<\infty$$
and obvious generalization for $p=\infty.$ For any $r,\ p\geq 1$, we will denote by $L^p(\Omega,\mathcal{F},\mathbb{P};L^r(0,T;X))$ the space of processes with values in $X$ such that \begin{enumerate}
    \item $u(\cdot,t)$ is progressively measurable.
    \item $u(\omega,t)\in X$ for almost all $(\omega,t)$ and \begin{align*}
        \mathbb{E}\left[\lVert u(\omega,\cdot)\rVert_{L^r(0,T;X)}^p\right]<+\infty.
    \end{align*}
    Obvious generalizations for $p=\infty$ or $r=\infty$. 
\end{enumerate}
 Set
$$H=\{f \in L^2(D;\mathbb{R}^2),\ \operatorname{div}f=0,\ f\cdot n|_{\partial D}=0\},\ V=H_0^1(D;\mathbb{R}^2)\cap H,\ D(A)=H^2(D;\mathbb{R}^2)\cap V.  $$ Moreover we introduce the vector space $$W=\{u\in V:\ \operatorname{curl}(u-\alpha^2\Delta u)\in L^2(D;\mathbb{R}^2) \} $$ with norm $\lVert{u}\rVert_W^2=\lVert u\rVert^2+\alpha^2\lVert \nabla u\rVert_{L^2(D;\mathbb{R}^2)}^2+\lVert \operatorname{curl}(u-\Delta u)\rVert_{L^2(D)}^2.$ It is well-known, see for example \cite{cioranescu1984existence}, that we can identify $W$ with the space 
$$\hat{W}=\{u\in H^3(D;\mathbb{R}^2)\cap V \}. $$ Moreover there exists a constant such that \begin{align}\label{equivalence H3-W}
    \lVert u\rVert_{H^3}^2\leq C\left(\lVert u\rVert^2+\lVert \nabla u\rVert_{L^2(D;\mathbb{R}^2)}^2+\lVert \operatorname{curl}(u-\Delta u)\rVert^2_{L^2(D)} \right).
\end{align} We denote by $\langle\cdot,\cdot\rangle$ and $\lVert\cdot\rVert$ the inner product and the norm in $H$ respectively. Other norms and scalar products will be denoted with the proper subscript. On $V$ we introduce the norm $\lVert u\rVert_V^2=\lVert u\rVert^2+\alpha^2\lVert \nabla u\rVert_{L^2(D;\mathbb{R}^2)}^2.$ We will shortly denote by $\lVert u\rVert_*=\lVert \operatorname{curl}(u-\alpha^2\Delta u)\rVert_{L^2(D)}.$ Obviously the following inequality holds for $u\in V$, where $C_p$ is the Poincarè constant associated to $D$,
\begin{align}\label{equivalence H1-V}
    \frac{\lVert u\rVert_V^2}{\alpha^2+C_p^2}\leq \lVert \nabla u\rVert_{L^2(D;\mathbb{R}^2)}^2\leq \frac{\lVert u\rVert_V^2}{\alpha^2}
\end{align}
Denote by $P$ the linear projector of $L^2\left(D;\mathbb{R}^2\right)$ on $H$ and define the unbounded linear operator $A:D(A)\subseteq H\rightarrow H$ by the identity \begin{align}\label{definition of A}
    \langle A v, w\rangle=\langle \Delta v, w \rangle_{L^2(D;\mathbb{R}^2)}
\end{align}
for all $v \in D(A),\ w \in H$. $A$ will be called the Stokes operator. It is well-known (see for example \cite{temam2001navier}) that $A$ is self-adjoint, generates an analytic semigroup of negative type on $H$ and moreover $V=D\left(\left(-A)^{1/2}\right)\right).$
Denote by $\mathbb{L}^{4}$ the space $L^{4}\left(  D,\mathbb{R}^{2}\right)
\cap H$, with the usual topology of $L^{4}\left(  D,\mathbb{R}^{2}\right)  $.
Define the trilinear, continuous form $b:\mathbb{L}^{4}\times V\times\mathbb{L}%
^{4}\rightarrow\mathbb{R}$ as%
\begin{align}\label{definition of b}
b\left(  u,v,w\right)  =\langle u, P(\nabla v w)\rangle.    
\end{align}

Now we introduce some assumptions on the stochastic part of the system.
\begin{hypothesis}\label{hypothesis well-posedness}
\begin{itemize}
    \item $K$ is a (possiblly countable) set of indexes. 
    \item $\sigma_k\in W^{1,\infty}(D;\mathbb{R}^2)\cap V$ satisfying \begin{align*}
        \sum_{k\in K}\lVert \sigma_k\rVert_{W^{1,\infty}}^2<+\infty.
    \end{align*}
    \item $u_0\in \cap_{p\geq 2}L^p(\mathcal{F}_0,W)$.
    \item $\{W^k_t\}_{k\in K}$ is a sequence of real, independent Brownian motions adapted to $\mathcal{F}_t$.
\end{itemize}
\end{hypothesis}
Let us consider the stochastic second-grade fluid equations below. Some physical motivations for the introduction of transport noise in fluid dynamic models can be found in \cite{flaWas}, \cite{holm2015variational}.
\begin{align*}
\begin{cases}
dv=(\nu \Delta u-\operatorname{curl}(v)\times u+\nabla p)dt+\sum_{k\in K} (\sigma_k\cdot \nabla u+\nabla\Tilde{p}_k) \circ dW^k_t\\
\operatorname{div}u=0\\
v=u-\alpha^2\Delta u\\
u|_{\partial D}=0\\ 
u(0)=u_0.
\end{cases}    
\end{align*}
We need to add the additional pressure term $\sum_{k\in K} \nabla\Tilde{p}_k \circ dW^k_t$, the so-called turubulent pressure, in the system above in order to deal with the fact that $\sum_{k\in K} \sigma_k\cdot \nabla u \circ dW^k_t$ is not divergence free, therefore an additional martingale term orthogonal to $H$ must be added to make the system feasible.\\
Introducing the Stokes operator, the previous equation can be rewritten as 
\begin{align}\label{Stratonovich}
\begin{cases}
d(u-\alpha^2 Au)=(\nu A u-P(\operatorname{curl}(v)\times u))dt+\sum_{k\in K} P(\sigma_k\cdot \nabla u) \circ dW^k_t\\
v=u-\alpha^2\Delta u\\
u(0)=u_0
\end{cases}    
\end{align}
or the corresponding It\^o form
\begin{align}\label{Ito}
\begin{cases}
d(u-\alpha^2 Au)& =(\nu A u-P(\operatorname{curl}(v)\times u))dt+\sum_{k\in K} P(\sigma_k\cdot \nabla u)  dW^k_t\\ & +\frac{1}{2}\sum_{k\in K} P(\sigma_k\cdot \nabla((I-\alpha^2A)^{-1}P(\sigma_k\cdot \nabla u)))dt\\
v=u-\alpha^2\Delta u\\
u(0)=u_0.
\end{cases}    
\end{align}
Indeed each of the Stratonovich integrals in equation \eqref{Stratonovich} can be rewritten, thanks to the Stratonovich-It\^o corrector associated to previous equation, in the following form: \begin{align}
    P(\sigma_k\cdot\nabla u)\circ dW^k_t&= P(\sigma_k\cdot\nabla u) dW^k_t+\frac{1}{2}d[P(\sigma_k\cdot\nabla u),W^k]_t \notag\\&= P(\sigma_k\cdot\nabla u) dW^k_t+\frac{1}{2}P\left(\sigma_k\cdot \nabla d[ u,W^k]_t\right) \notag\\ & =P(\sigma_k\cdot\nabla u) dW^k_t+\frac{1}{2}P\left(\sigma_k\cdot \nabla d\left[ \int_0^.(I-\alpha^2A)^{-1}\sum_{j\in K}P(\sigma_j\cdot\nabla u)dW^j_s,W^k\right]_t\right) \notag\\&= P(\sigma_k\cdot\nabla u) dW^k_t+\frac{1}{2}P\left(\sigma_k\cdot \nabla \left((I-\alpha^2 A)^{-1}P(\sigma_k\cdot \nabla u)\right)\right) dt. \label{Ito Strat corrector}
\end{align}
We denote by $F(u)=\frac{1}{2}\sum_{k\in K} P(\sigma_k\cdot \nabla((I-\alpha^2A)^{-1}P(\sigma_k\cdot \nabla u)))$ and $G^k(u)=P(\sigma_k\cdot \nabla u)$. By Corollary \ref{regularity of G^k, F} below \begin{align*}
    F\in \mathcal{L}(V;H\cap H^1(D;\mathbb{R}^2)),\ \ G^K\in \mathcal{L}(V;H).
\end{align*} 
\begin{definition}\label{weak solution}
A stochastic process with weakly continuous trajectories with values in $W$
is a weak solution of equation (\ref{Ito}) if
$$
u \in L^p(\Omega,\mathcal{F},\mathbb{P};L^{\infty}(0,T;W)),\ \forall p\geq 2.
$$
and $\mathbb{P}-a.s.$ for every $t\in [0,T]$ and $\phi \in W$  we have
\begin{align*}
    &\langle u(t)-u_0,\phi\rangle_V+\int_0^t \nu \langle \nabla u(s),\nabla \phi\rangle_{L^2(D;\mathbb{R}^2)}+ \langle \operatorname{curl}(u(s)-\alpha^2\Delta u(s))\times u(s), \phi\rangle_{L^2(D)} ds\\ & =\int_0^t \langle F(u(s)),\phi \rangle ds +\sum_{k\in K}\int_0^t \langle G^k(u(s)),\phi\rangle dW^k_s.
\end{align*}
\end{definition}

\begin{theorem}\label{Thm well posed}
Under Hypothesis \ref{General Hypothesis}-\ref{hypothesis well-posedness}, equation (\ref{Ito}) has a unique solution in the sense of definition \ref{weak solution}. Moreover, almost surely, the stochastic process $u$ has $V$ continuous paths.
\end{theorem}
\begin{remark}\label{weak integrability}
    Actually we can weaken the integrability assumption of $u_0$ with respect to $\mathbb{P}$ in order to get less integrable solution, but regular enough to prove that the inviscid limit holds. Indeed $u_0\in L^4(\mathcal{F}_0,W)$ is the minimal assumption to prove either the well-posedness, see \cite{breckner2000galerkin} and Section \ref{Sec Galerkin approximation} below, and the the inviscid limit, see Sections \ref{energy estimates} and \ref{proof of main thm}. However, we prefer to not stress this assumption in order to make our results comparable to \cite{razafimandimby2012strong}.
\end{remark}
As stated in Section \ref{intro}, the proof of Theorem \ref{Thm well posed} will be the object of Section \ref{well posed}. Usually, in stochastic analysis, the well-posedness of a stochastic partial differential equation is obtained considering some approximating sequence, $\{u^N\}_{N\in \mathbb{N}}$, which solves an approximate equation in the original probability space and showing the tightness of their law in some spaces of functions. Then, by Prokhorov's theorem and Skorokhod's representation theorem, one can find an auxiliary probability space and a solution of the limit equation in this auxiliary probability space, $u$. Lastly, by a Gyongy-Krylov argument, one can recover that the limit process belongs to the original probability space and that the approximating sequence converge in probability to $u$.  See \cite{bensoussan1995stochastic}, \cite{capinski1994stochastic}, \cite{flaWas} for some examples of this method. Here, we follow a different, perhaps, more direct approach introduced by Breckner in \cite{breckner2000galerkin} for Navier-Stokes equations with multiplicative noise with particular regularity properties, but well-suited to treat transport noise, which a priori does not satisfy the general assumptions of \cite[Section 2]{breckner2000galerkin}. This approach uses, in particular, the properties of stopping times, some basic convergence principles from functional analysis and some properties of fluid dynamic non-linearities. Therefore, even if the results of \cite{breckner2000galerkin} were related to Navier-Stokes equations, this approach can be applied also to other fluid dynamic models, see \cite{razafimandimby2012strong}, \cite{carigi2023dissipation} for some examples to different fluid dynamic systems. An important byproduct of this way of proceed is that the approximations converge in mean square to the solution of the second-grade fluid equations, see Theorem \ref{theorem convergence galerkin} below. This fact will be crucial in order to obtain some apriori estimates on the solution, see Lemma \ref{thm energy inequality} below. 

Now we move to consider the inviscid limit problem and introduce a new set of hypotheses.
\begin{hypothesis}\label{Hypothesis inviscid limit}
\begin{itemize}
    \item $\nu=O(\alpha^2),\ \Tilde{\nu}=O(\alpha^2)$.
    \item $\overline{u}_0\in H^s(D;\mathbb{R}^2)\cap H$ for some $s\geq 3.$
    \item \begin{align}
    \mathbb{E}\left[\lVert u^{\alpha}_0-\bar{u}_0\rVert^2\right]\rightarrow 0 \label{L^2 conv initial};\\
    \mathbb{E}\left[\alpha^2\lVert \nabla u^{\alpha}_0\rVert_{L^2(D;\mathbb{R}^2)}^2\right]=o(1) \label{H^1 behav initial};\\
    \mathbb{E}\left[\alpha^6\lVert u^{\alpha}_0\rVert_{H^3(D;\mathbb{R}^2)}^2\right]=O(1) \label{H^3 behav initial}.
\end{align}
    
\end{itemize}
\end{hypothesis}
Let us consider the family of equations
\begin{align}\label{Ito Scaled}
\begin{cases}
d(u^{\alpha}-\alpha^2 Au^{\alpha})& =(\nu A u^{\alpha}-P(\operatorname{curl}(v^{\alpha})\times u^{\alpha}))dt+\sqrt{\Tilde{\nu}}\sum_{k\in K} P(\sigma_k\cdot \nabla u^{\alpha})  dW^k_t\\ & +\frac{\Tilde{\nu}}{2}\sum_{k\in K} P(\sigma_k\cdot \nabla((I-\alpha^2A)^{-1}P(\sigma_k\cdot \nabla u^{\alpha})))dt\\
v^{\alpha}=u^{\alpha}-\alpha^2\Delta u^{\alpha}\\
u^{\alpha}(0)=u_0^{\alpha},
\end{cases}    
\end{align}
where $\sigma_k$ are independent from $\nu,\  \Tilde{\nu},\ \alpha$ and $u^{\alpha}_0$ are random variable satisfying the assumptions of Theorem \ref{Thm well posed}. Energy relations and the behavior of the $H^3$ norm of $u^{\alpha}$ play a crucial role in the analysis of the inviscid limit in the deterministic framework, see Equation (3.2) and Equation (3.7) in \cite{lopes2015approximation}. If we want to have some hope of replicating the approach of \cite{lopes2015approximation} we need some estimates in that direction. This is exactly what happens. Indeed, under Hypothesis \ref{Hypothesis inviscid limit},  Equation (3.2) and Equation (3.7) in \cite{lopes2015approximation} continue to hold in the stochastic framework, see Lemma \ref{thm energy inequality} below. Therefore there is some hope to generalize the results of \cite{lopes2015approximation}, \cite{lopes2015convergence} to our stochastic framework. Now, let us consider the Euler equations
\begin{align}\label{Euler equations}
    \begin{cases}
        \partial_t \bar{u}+\nabla \bar{u}\cdot \bar{u}+\nabla p=0 \ (x,t)\in D\times(0,T)\\ 
        \operatorname{div}\bar{u}=0 \\ 
        \bar{u}\cdot n|_{\partial D}=0 \\ 
        \bar{u}(0)=\bar{u}_0.
    \end{cases}
\end{align}\label{Euler}
\begin{definition}\label{Well Posed Euler }
Given $\bar{u}_0\in H,$ we say that $\bar{u}\in C(0,T;H)$ is a weak solution of equation (\ref{Euler equations}) if for every $\phi \in C^{\infty}_0([0,T]\times D)\cap C^1([0,T];H)$ \begin{align*} \langle \bar{u}(t),\phi(t)\rangle=\langle \bar{u}_0,\phi (0)\rangle+\int_0^t\langle \bar{u}(s),\partial_s\phi(s)\rangle ds +\int_0^t b(\bar{u}(s),\phi(s),\bar{u}(s)) ds
\end{align*}
for every $t \in[0,T]$ and the energy inequality \begin{align*}
    \lVert{\bar{u}(t)}\rVert^2\leq\lVert{\bar{u}_0}\rVert^2
\end{align*}
holds.
\end{definition}
For what concern the well posedness of the Euler equations the following results hold true, see \cite{kato1984nonlinear}, \cite{temam1974euler}.
\begin{theorem}\label{Kato classiche}
Fix $T>0,\ s\geq 3$. Let $\bar{u}_0\in H^s(D;\mathbb{R}^2)\cap H$. Then there exist a unique weak solution of (\ref{Euler equations}) with initial condition $\bar{u}_0$ such that $$ \bar{u}\in C([0,T];H^s(D;\mathbb{R}^2))\cap C^1([0,T];H^{s-1}(D;\mathbb{R}^2))$$
and $\lVert{ \bar{u}(t)}\rVert=\lVert \bar{u}_0\rVert,\ \forall t\in [0,T].$
\end{theorem}
Now we can state our main Theorem. According to the analysis started in \cite{galeati2020convergence} and continued, recently, in \cite{flandoli2021quantitative}, \cite{flandoli2021scaling} the influence of the transport noise on the averaged solution is related to the $\ell^2$ norm of its coefficients, therefore we expect that the solution of equation (\ref{Ito Scaled}) converges to the solution of the Euler equations with null forcing term. 
\begin{theorem}\label{Main thm}
Under Hypotheses \ref{General Hypothesis}-\ref{hypothesis well-posedness}-\ref{Hypothesis inviscid limit}, calling $u^{\alpha}$ the solution of (\ref{Ito Scaled}) and $\bar{u}$ the solution of (\ref{Euler equations}), then \begin{align*}
    \lim_{\alpha \rightarrow 0}\mathbb{E}\left[\sup_{t\in [0,T]}\lVert u^{\alpha}(t)-\bar{u}(t)\rVert^2\right]= 0.
\end{align*}
\end{theorem}
\begin{remark}
If $\overline{u}_0\in H\cap H^1(D;\mathbb{R}^2)$, the existence of a family $u_0^{\alpha}$ satisfying equations \eqref{L^2 conv initial},\eqref{H^1 behav initial},\eqref{H^3 behav initial} is guaranteed by Proposition 1 of \cite{lopes2015convergence}.
\end{remark}
\begin{remark}
Due to the poor regularity of the coefficients $F$ and $G^k$, equations (\ref{Ito}) are not guaranteed to be well-posed for grant from the results of \cite{razafimandimby2012strong}. Indeed, neither $F$ nor $G^k$ satisfy the assumptions of \cite{razafimandimby2010weak} or \cite{razafimandimby2012strong}. However, due to relation \eqref{trilinear G^{k,N}} and the good estimates of Corollary \ref{regularity of G^k, F}, we will be able to prove in Lemma \ref{energy in V} and Lemma \ref{energy in W} the same, actually stronger, energy estimates that are available in \cite{razafimandimby2012strong}. These and Lemma \ref{crucial lemma} are the main ingredients in order to prove the well-posedness of system \eqref{Ito}. On the contrary the well-posedness in the case of additive noise is completely solved by the results of \cite{razafimandimby2012strong}, thus in Section \ref{additive noise} we will only explain some remarks about the inviscid limit and the well-posedness in the additive noise framework.
\end{remark}
\begin{remark}
Either Theorem \ref{Thm well posed} and Theorem \ref{Main thm} continue to hold considering $\nu=0$. We will give the proof of all the statements below in full details considering the case $\nu>0$. However if something in the proof changes considering $\nu=0$ we will explain in a remark at the end of each proof what we need to change in order to deal to the other case. 
\end{remark}
\begin{remark}
The arbitrariness in the choice of the parameters $\nu$ and $\Tilde{\nu}$ allows us to generalize to this stochastic framework, via  Theorem \ref{Main thm}, some results of \cite{lopes2015convergence} and \cite{lopes2015approximation}. As a byproduct of its proof we obtain that under Hypotheses \ref{General Hypothesis}-\ref{hypothesis well-posedness}-\ref{Hypothesis inviscid limit} \begin{align*}
    \alpha^2\operatorname{sup}_{t\in [0,T]}\mathbb{E}\left[\lVert \nabla u^{\alpha}(t)\rVert_{L^2(D;\mathbb{R}^2)}^2\right]\rightarrow 0.
\end{align*} Moreover, considering $\Tilde{\nu}=\nu>0$ we recover the scaling introduced by Kuksin in \cite{kuksin2004eulerian} which is relevant for the inviscid limit at the level of invariant measures.
\end{remark}
\begin{remark}
The results of these notes are in a certain sense complementary to what we obtained in \cite{luongo2021inviscid}. In \cite{luongo2021inviscid} we
required poor regularity on the initial conditions of the Euler equations and the Navier-Stokes equations but we
got a conditioned result. On the contrary, in these notes we require strong regularity on the initial conditions of
the two problems and a special type of convergence of the initial conditions but we arrive at a not conditioned
result.
\end{remark}
\begin{remark}
The assumption on $\nu=O(\alpha^2)$ is hidden in equation (\ref{Ito Scaled}). For high frequencies $\Delta u$ is a damping term in equation (\ref{Ito Scaled}). In fact, for high frequencies $v\approx -\alpha^2 \Delta u$, thus the equation becomes, formally, \begin{align*}
    -\alpha^2 \partial_t \Delta u-\nu \Delta u+\dots=0.
\end{align*}
Asking $\nu=O(\alpha^2)$, means requiring that the damping coefficient does not blow-up.
\end{remark}
We conclude this section with few notations that will be adopted:
by $C$ we will denote several constant independent from $\nu$, $\alpha^2$ and $\sigma_k$, perhaps changing value line by line. In the case $C$ will depends by $\nu,\alpha$ or $\sigma_k$ we will add the dependence as a subscript. Sometimes we will use the notation $a \lesssim b$, if it exists a constant independent from $\nu$ and $\alpha^2$ such that $\alpha \leq C b$. In order to simplify the notation we will denote Sobolev spaces by $H^s$, forgetting domain and range.

\section{Well-Posedness}\label{well posed}
\subsection{Preliminaries}\label{preliminaries}
Before starting with the analysis of equation (\ref{Ito}), we need to recall some preliminaries results on the nonlinear term in the second-grade fluid equations, the Stokes operator $A$ and the embedding between $W$ and $V$. We will consider the Hilbert triple $$ W\hookrightarrow V\hookrightarrow W^*$$\\
We start recalling in a single lemma some classical facts on the nonlinear part of equation \eqref{second-grade system}. We refer to \cite{cioranescu1984existence},\cite{razafimandimby2010weak}, \cite{razafimandimby2012strong} for the proof of the various statements.
\begin{Lemma}\label{nonlinearity}
For any smooth, divergence free $\phi,\ v,\ w$ the following relation holds \begin{align}\label{equvalence hatB and b}
    \langle \operatorname{curl}\phi \times v, w\rangle_{L^2}=b(v,\phi,w)-b(w,\phi,v).
\end{align}
Moreover for $u,\ v,\ w$ the following inequalities hold \begin{align}
    \lvert \langle \operatorname{curl}(u-\alpha^2\Delta u) \times v, w\rangle_{L^2}\rvert\leq C \lVert u\rVert_{H^3}\lVert v\rVert_V \lVert w\rVert_W \label{inequality trilinear 1}\\
    \lvert \langle \operatorname{curl}(u-\alpha^2\Delta u) \times u, w\rangle_{L^2}\rvert\leq C \lVert u\rVert^2_V \lVert w\rVert_W\label{inequality trilinear 2}
\end{align}
Therefore there exists a bilinear operator $\hat{B}:W\times V\rightarrow W^*$ such that \begin{align}\label{definition hatB}
    \langle \hat{B}(u,v),w\rangle_{W^*,W}=\langle P(\operatorname{curl}(u-\alpha^2\Delta u)\times v),w \rangle
\end{align}
which satisfies for $u\in V,\ v\in W$
\begin{align}
    \lVert \hat{B}(v,u)\rVert_{W^*}\leq C \lVert u\rVert_V\lVert v\rVert_W \label{inequality bilinear 1}\\ 
    \lVert \hat{B}(u,u)\rVert_{W^*}\leq C \lVert u\rVert_V^2\label{inequality bilinear 2}.
\end{align}
Lastly, for $u\in W,\ v\in V, \ w \in W$
\begin{align}\label{antisimmetry hatB}
    \langle \hat{B}(u,v),w\rangle_{W^*,W}=-\langle \hat{B}(u,w),v\rangle_{W^*,W}.
\end{align}
\end{Lemma}
We need a basis orthonormal either in $W$ and in $V$ in order to deal with the Galerkin approximation of equation \eqref{Ito system}. The existence of such basis is guaranteed by the lemma below. The first part is a consequence of spectral theorem for self-adjoint compact operators stated in \cite{reed2012methods}, we refer to \cite{cioranescu1997weak} for the proof of the second part.
\begin{Lemma}\label{lemma eigenfunctions}
The injection of $W$ into $V$ is compact. Let $I$ be the isomorphism of $W^*$ onto $W$, then the restriction of $I$ to $V$ is a continuous compact operator into itself. Thus, there exists a sequence $e_i$ of elements of $W$ which forms an orthonormal basis in $W$, and an orthogonal basis in $V$. This sequence verifies:

\begin{align}
    \textit{for any   } v\in W \ \ \langle v, e_i\rangle_W=\lambda_i \langle v,e_i\rangle_V
\end{align} where $\lambda_{i+1}> \lambda_i> 0,\  i= 1,2,\dots$. Thus $\sqrt{\lambda_i}e_i$ is an orthonormal basis of $V$.
Moreover $e_i$ belong to $H^4(D;\mathbb{R}^2)$. 
\end{Lemma}
We will use also some properties of the projection operator $P$ and the solution map of the Stokes operator. We refer to \cite{temam2001navier} for the proof of the lemmas below.
\begin{Lemma}\label{lemma projection}
The restriction of the projection operator $P:L^2(D;\mathbb{R}^2)\rightarrow H$ to $H^r(D;\mathbb{R}^2)$ is a continuous and linear map between $H^r(D;\mathbb{R}^2)$ and itself.
\end{Lemma}
\begin{Lemma}\label{Stokes}
 Let $f\in H^m(D;\mathbb{R}^2)$. Then, there exists a unique couple $(u,p)$, with $p$ defined up to an additive constant, solution of \begin{align*}
    u-\alpha^2\Delta u+\nabla p=f\\ 
    \operatorname{div}u=0\\
    u|_{\partial D}=0.
\end{align*}
Moreover $u=(I-\alpha^2A)^{-1}f\in H^{m+2}(D;\mathbb{R}^2),\ p\in H^{m+1}(D)$,\begin{align*}
\lVert u\rVert_{H^{m+2}}+\lVert p\rVert_{H^{m+1}}\leq C\lVert f\rVert_{H^m}.    
\end{align*}

\end{Lemma}

\begin{Lemma}\label{eigenfunction Stokes}
The injection of $V$ in $H$ is compact. Thus there exists a sequence $\Tilde{e}_i$ of elements of $H$ which forms an orthonormal basis in $H$ and an orthogonal basis in $V$. This sequence verifies \begin{align*}
    -A\Tilde{e}_i=\Tilde{\lambda}_i\Tilde{e}_i
\end{align*}
where $\Tilde{\lambda}_{i+1}>\Tilde{\lambda}_{i}>0,\ i=1,2,\dots$. Moreover $\Tilde{\lambda}_i\rightarrow +\infty$. Lastly $\Tilde{e}_i\in C^{\infty}(\overline{D};\mathbb{R}^2)$  under our assumptions on $D$ 
\end{Lemma}
Combining Lemma \ref{lemma eigenfunctions} and Lemma \ref{Stokes} above, it follows that for each $f\in H^1$, $i \in\mathbb{N}$
\begin{align}
\langle (I-\alpha^2A)^{-1}f, e_i\rangle_W=\lambda_i\langle (I-\alpha^2A)^{-1}f, e_i\rangle_V=\lambda_i \langle f, e_i\rangle \label{result inner product eigenfucntions}.
\end{align}
Moreover, Lemmas \ref{lemma projection}, \ref{Stokes}, \ref{eigenfunction Stokes} above allow us to prove some useful estimates that will be exploited along the paper. We will need Corollary \ref{regularity of G^k, F} in order to evaluate the regularity of the linear operators appearing in equation \eqref{Ito system}. Instead we will need Lemma \ref{behavior G^k, F} in order to quantify explicitly the dependence from $\alpha$ in several embeddings and operators. This will be crucial in Section \ref{energy estimates} and Section \ref{proof of main thm}.\\
We recall first that by Poincaré inequality, equation \eqref{equivalence H3-W}, triangle inequality and equation \eqref{equivalence H1-V} the following relations hold:
\begin{align}
   & \lVert u \rVert_{H^3}^2 \leq C( \lVert \nabla u\rVert_{L^2}^2+\lVert \operatorname{curl}(u-\Delta u)\rVert_{L^2}^2)\leq C\left(\frac{\alpha^2+1}{\alpha^2}\right)^2\lVert \nabla u\rVert_{L^2}^2+\frac{C}{\alpha^4}\lVert \operatorname{curl}( u-\alpha^2\Delta u)\rVert_{L^2}^2 \label{scaling H^3}\\ & \lVert \nabla u\rVert_{L^2}^2\leq \frac{\lVert u\rVert_V^2}{\alpha^2}\label{scaling H^1}.
\end{align}
\begin{Lemma}\label{behavior G^k, F}

Let  $h\in H,\ u\in V,\ w\in V\cap H^2$, then
\begin{align}
        \lVert G^k(u)\rVert &\leq \lVert \sigma_k\rVert_{L^{\infty}}\lVert \nabla u\rVert_{L^2} \leq \lVert \sigma \rVert_{L^{\infty}}\frac{\lVert u\rVert_V}{\alpha}, \label{behave G^k V H}\\ 
    \lVert \nabla G^k(w)\rVert_{L^2} &\leq C\lVert \sigma_k\rVert_{W^{1,\infty}}{\lVert w\rVert_{H^2}},\label{behave G^k H^2 H^1}\\
    \lVert (I-\alpha^2A)^{-1}h\rVert & \leq \lVert h\rVert \label{norm HD(A)},\\ \lVert (-A)^{1/2}(I-\alpha^2A)^{-1}h\rVert&\leq\frac{1}{2\alpha} \lVert h\rVert,\label{norm HV}\\ \lVert -A(I-\alpha^2A)^{-1}h\rVert &\leq\frac{1}{\alpha^2} \lVert h\rVert\label{norm HH},\\ 
    \lVert (I-\alpha^2A)^{-1}(P(\sigma_k\cdot\nabla w))\rVert_W &\leq C\lVert \sigma_k\rVert_{W^{1,\infty}}\lVert w\rVert_{H^2}  \label{behave tilde G^k}.
\end{align}
Therefore, if $u\in V$ the following inequalities hold true
\begin{align}
    \lVert (I-\alpha^2A)^{-1}P(\sigma_k\cdot \nabla u)\rVert &\leq \lVert \sigma_k\rVert_{L^{\infty}}\lVert \nabla u\rVert_{L^2}\leq  \lVert \sigma_k\rVert_{L^{\infty}}\frac{\lVert u\rVert_{V}}{\alpha}\label{G^k H H^2},\\ 
     \lVert \nabla(I-\alpha^2A)^{-1}P(\sigma_k\cdot \nabla u)\rVert_{L^2} & \leq \frac{\lVert \sigma_k\rVert_{L^{\infty}}}{2\alpha}\lVert \nabla u\rVert_{L^2} \label{G^k H H^1},\\
      \lVert P(\sigma_k\cdot\nabla((I-\alpha^2A)^{-1}P(\sigma_k\cdot\nabla u)))\rVert & \leq \frac{\lVert \sigma_k\rVert_{L^{\infty}}^2}{2\alpha}\lVert \nabla u\rVert_{L^2}\label{preliminary F1}.\\
        \lVert P(\sigma_k\cdot\nabla((I-\alpha^2A)^{-1}P(\sigma_k\cdot\nabla u)))\rVert_{H^1} & \leq \frac{C\lVert \sigma_k\rVert_{L^{\infty}}\lVert \sigma_k\rVert_{W^{1,\infty}}\lVert \nabla u\rVert_{L^2} }{\alpha^2}\label{preliminary F2}.
\end{align}
\end{Lemma}
\begin{proof}
Inequalities \eqref{behave G^k V H}, \eqref{behave G^k H^2 H^1} are trivial. Indeed, by Lemma \ref{lemma projection} it holds 
\begin{align*}
    \lVert G^k(u)\rVert=\lVert P(\sigma_k\cdot \nabla u)\rVert\leq \lVert \sigma_k\cdot\nabla u\rVert_{L^2}\leq \lVert \sigma_k\rVert_{L^{\infty}}{\lVert \nabla u\rVert_{L^2}}\leq \lVert \sigma_k\rVert_{L^{\infty}}\frac{\lVert u\rVert_V}{\alpha}.
    \end{align*}
\begin{align*}
    \lVert \nabla G^k(w)\rVert_{L^2}&=\lVert \nabla P(\sigma_k\cdot \nabla w)\rVert_{L^2}\\ &\leq C \lVert  \sigma_k\cdot \nabla w\rVert_{H^1}\\ & \leq C \lVert \sigma_k\rVert_{W^{1,\infty}}\lVert \nabla w\rVert_{H^1}\\ & \leq C \lVert \sigma_k\rVert_{W^{1,\infty}}\lVert  w\rVert_{H^2}.
\end{align*}
In order to prove inequalities \eqref{norm HD(A)}, \eqref{norm HV}, \eqref{norm HH} we exploit the Fourier decomposition $h=\sum_{i\in \mathbb{N}}\langle h,\Tilde{e}_i\rangle \Tilde{e}_i$. Therefore it holds
\begin{align*}
    \lVert (I-\alpha^2A)^{-1}h\rVert^2=\sum_{i\in\mathbb{N}}\frac{\langle h,\Tilde{e}_i\rangle^2}{(1+\alpha^2\Tilde{\lambda}_i)^2}\leq \lVert h\rVert^2,
\end{align*}
\begin{align*}
    \lVert (-A)^{1/2}(I-\alpha^2A)^{-1}h\rVert^2=\sum_{i\in\mathbb{N}}\frac{\Tilde{\lambda}_i}{(1+\alpha^2\Tilde{\lambda}_i)} \langle h,\Tilde{e}_i\rangle^2\leq\frac{1}{4\alpha^2} \lVert h\rVert^2,
\end{align*}
 \begin{align*}
     \lVert -A(I-\alpha^2A)^{-1}h\rVert=\sum_{i\in\mathbb{N}}\frac{\Tilde{\lambda}_i^2}{(1+\alpha^2\Tilde{\lambda}_i)} \langle h,\Tilde{e}_i\rangle^2\leq\frac{1}{\alpha^4} \lVert h\rVert^2.
\end{align*}  
For what concerns inequality \eqref{behave tilde G^k}, by definition of the norm in the space $W$ it holds \begin{align*}
    \lVert (I-\alpha^2A)^{-1}(P(\sigma_k\cdot\nabla w))\rVert_W^2=\lVert (I-\alpha^2A)^{-1}(P(\sigma_k\cdot\nabla w))\rVert_V^2+\lVert \operatorname{curl}((I-\alpha^2 \Delta)(I-\alpha^2A)^{-1}(P(\sigma_k\cdot\nabla w)))\rVert_{L^2}^2.
\end{align*}
 From Lemma \ref{Stokes}, we know that \begin{align}\label{preliminary G1}
     \lVert (I-\alpha^2A)^{-1}(P(\sigma_k\cdot\nabla w))\rVert_V^2&=\langle P(\sigma_k\cdot\nabla w),(I-\alpha^2A)^{-1}(P(\sigma_k\cdot\nabla w))\rangle\notag\\ & \leq \lVert \sigma_k\rVert_{L^{\infty}}\lVert \nabla w\rVert_{L^2}\lVert (I-\alpha^2A)^{-1}(P(\sigma_k\cdot\nabla w))\rVert_V\notag \\ & \leq \lVert \sigma_k\rVert_{L^{\infty}}^2\lVert \nabla w\rVert_{L^2}^2,\end{align}
     \begin{align}\label{preliminary G2}
    \lVert \operatorname{curl}\left((I-\alpha^2\Delta)(I-\alpha^2A)^{-1}(P(\sigma_k\cdot\nabla w))\right)\rVert_{L^2}=\lVert \operatorname{curl }(P(\sigma_k\cdot\nabla w))\rVert_{L^2}\leq C\lVert \sigma_k\rVert_{W^{1,\infty}}\lVert w\rVert_{H^2}
    \end{align}
Combining \eqref{preliminary G1} and \eqref{preliminary G2}, inequality \eqref{behave tilde G^k} follows.

Combining relation \eqref{behave G^k V H} with relations \eqref{norm HD(A)} and \eqref{norm HV}, inequalities \eqref{G^k H H^2} and \eqref{G^k H H^1} follow immediately. 
Let us now prove equation \eqref{preliminary F1}. By H\"older's inequality and relation \eqref{behave G^k H^2 H^1} we have
\begin{align*}
     \lVert P(\sigma_k\cdot\nabla((I-\alpha^2A)^{-1}P(\sigma_k\cdot\nabla u)))\rVert &\leq  \lVert\sigma_k\cdot\nabla((I-\alpha^2A)^{-1}P(\sigma_k\cdot\nabla u))\rVert_{L^2}\notag\\ & \leq \lVert \sigma_k\rVert_{L^{\infty}}\lVert  \nabla((I-\alpha^2A)^{-1}P(\sigma_k\cdot\nabla u))\rVert_{L^2}\notag\\ & \leq \frac{\lVert \sigma_k\rVert_{L^{\infty}}^2}{2\alpha}\lVert \nabla u\rVert_{L^2}.
\end{align*}

For what concerns the last one, by Lemma \ref{lemma projection}, \ref{Stokes} and relations \eqref{norm HH} it holds
    \begin{align*}
        \lVert P(\sigma_k\cdot\nabla((I-\alpha^2A)^{-1}P(\sigma_k\cdot\nabla u)))\rVert_{H^1} &\leq C  \lVert\sigma_k\cdot\nabla((I-\alpha^2A)^{-1}P(\sigma_k\cdot\nabla u))\rVert_{H^1}\notag\\ & \leq C\lVert \sigma_k\rVert_{W^{1,\infty}}\lVert \nabla((I-\alpha^2A)^{-1}P(\sigma_k\cdot\nabla u)) \rVert_{H^1}\notag\\ & \leq C\lVert \sigma_k\rVert_{W^{1,\infty}}\lVert (I-\alpha^2A)^{-1}P(\sigma_k\cdot\nabla u) \rVert_{H^2}\notag\\ & \leq C\lVert \sigma_k\rVert_{W^{1,\infty}}\lVert A(I-\alpha^2A)^{-1}P(\sigma_k\cdot\nabla u) \rVert\notag\\ & \leq \frac{C}{\alpha^2}\lVert \sigma_k\rVert_{W^{1,\infty}}\lVert P(\sigma_k\cdot\nabla u) \rVert\notag\\ & \leq \frac{C\lVert \sigma_k\rVert_{L^{\infty}}\lVert \sigma_k\rVert_{W^{1,\infty}}\lVert \nabla u\rVert_{L^2} }{\alpha^2}.
    \end{align*}

\end{proof}
\begin{corollary}\label{regularity of G^k, F}
\begin{align*}
    G^k\in \mathcal{L}(V;H),\quad F\in \mathcal{L}(V;H\cap H^1(D;\mathbb{R}^2)).
\end{align*}
In particular 
\begin{align}
    \lVert G^k(u)\rVert & \leq \lVert \sigma \rVert_{L^{\infty}}\frac{\lVert u\rVert_V}{\alpha}, \label{behave G^k V H bis}\\
      \lVert F(u)\rVert_{H}&\leq \frac{1}{2\alpha} \sum_{k\in K}\lVert \sigma_k\rVert_{L^{\infty}}^2\lVert \nabla u\rVert_{L^2} \label{behave F V H}\\
    \lVert F(u)\rVert_{H^1}&\leq \frac{C }{\alpha^2} \sum_{k\in K}\lVert \sigma_k\rVert_{L^{\infty}}\lVert \sigma_k\rVert_{W^{1,\infty}}\lVert \nabla u\rVert_{L^2} \label{behave F V H1}.
\end{align}
\end{corollary}
Lastly we recall two technical tools used in the proof of Theorem \ref{Main thm}. We refer to \cite{galdi2011introduction} for the proof of the interpolation inequality and to \cite{scheutzow2013stochastic} for the proof of the stochastic Gr\"onwall's Lemma.
\begin{theorem}Each function $f\in H^2$ satisfies the following inequality:
\begin{align}\label{interpolation estimate}
\lVert f\rVert_{H^1}\leq C   \lVert f\rVert_{L^2}^{1/2}\lVert f\rVert_{H^2}^{1/2}.
\end{align}
\end{theorem}
\begin{theorem}\label{stochastic gronwall }
Let $Z(t)$ and $H(t)$ be continuous, nonnegative, adapted processes, $\psi(t) $ a nonnegative deterministic function and $M(t)$ a continuous local martingale such that \begin{align*}
    Z(t)\leq \int_0^t\psi(s)Z(s)ds +M(t)+H(t) \quad \forall t\in [0,T].
\end{align*}
Then $Z(t)$ satisfies the following inequality
\begin{align}
    \mathbb{E}[Z(t)]\leq exp\left(\int_0^t \psi(s) ds\right)\mathbb{E}\left[\operatorname{sup}_{r\in [0,s]}H(s)\right].
\end{align}
\end{theorem}
\subsection{Galerkin Approximation and Limit Equations}\label{Sec Galerkin approximation}
Let $W^N=\operatorname{span}\{e_1,\dots,\ e_N\}\subseteq W$ and $P^N:W\rightarrow W^N$ the orthogonal projector. We start looking for a finite dimension approximation of the solution of equation (\ref{Ito}). We define $$ u^N(t)=\sum_{i=1}^N c_{i,N}(t)e_i.$$ 
The $c_{i,N}$ have been chosen in order to satisfy $\forall e_i,\ \ 1\leq i\leq N$
\begin{align*}
    \langle u^N(t), e_i\rangle_V-  \langle u^N_0, e_i\rangle_V&=\nu \int_0^t \langle \nabla u^N(s), \nabla e_i\rangle_{L^2} ds- \int_0^t b(u^N(s), u^N(s)-\alpha^2 \Delta u^N(s),e_i)ds\\ &-\alpha^2 \int_0^t b(e_i, \Delta u^N(s),u^N(s))ds+\int_0^t \langle F^N(s),e_i \rangle ds \\ &+ \sum_{k\in K} \int_0^t \langle G^{k,N}(s),e_i \rangle  dW^k_s \ \ \mathbb{P}-a.s.
\end{align*}
where $u_0^N=\sum_{i=1}^N\langle u_0,e_i\rangle_W e_i$, $F^N(s)=F(u^N(s))$ and $G^{k,N}(s)=G^k(u^N(s))$. The local well-posedness of this equation follows from classical results about stochastic differential equations with locally Lipshitz coefficients, see for example \cite{karatzas2012brownian},\cite{skorokhod1982studies}. The global well-posedness follows from the a priori estimates in Lemma \ref{energy in V}, \ref{energy in W}.

\begin{Lemma}\label{energy in V}
Assuming Hypothesis \ref{hypothesis well-posedness}, the following relations hold:
\begin{itemize}
    \item The It\^o's formula
\begin{align}\label{ito galerkin}
    d \lVert u^N\rVert_V^2=-2\nu \lVert \nabla u^N\rVert_{L^2}^2 dt-\sum_{k\in K} b(\sigma_k, u^N,(I-P^N)(I-\alpha^2 A)^{-1}P(\sigma_k\cdot\nabla u^N))dt.
\end{align}

\item The inequality below holds uniformly in $N$
\begin{align}\label{norm V galerkin}
    \mathbb{E}\left[\operatorname{sup}_{t\in [0,T]}\lVert u^N(t)\rVert_V^p\right]\leq C_{p,\alpha,\{\sigma_k\}_{k\in K}},\ \ \forall p\geq 1.
\end{align}
\end{itemize}
\end{Lemma}
\begin{proof}
If we apply the It\^o's formula to $\sum_{i=1}^N \lambda_i\langle u^N(t),e_i\rangle_V^2$, we get \begin{align}\label {Ito Norm V galerkin}
    \lVert u^N(t)\rVert_V^2+2\nu\int_0^t \lVert \nabla u^N(s)\rVert_{L^2}^2 ds& =\lVert u^N_0\rVert_V^2+2\int_0^t \langle F^N(s),u^N(s)\rangle ds+\sum_{i=1}^N\sum_{k\in K} \lambda_i \int_0^t \langle G^{k,N}(s),e_i\rangle^2 ds\notag\\ &+ 2\sum_{k=1}^N \int_0^t \langle G^{k,N}(s), u^N(s)\rangle dW^k_s.
\end{align}
In the last relation we exploited the fact that $b(u^N(s), u^N(s), u^N(s))=b(u^N(s),\Delta u^N(s), u^N(s))=0$.
Now we observe that for each $k$, $\langle G^{k,N}(s), u^N(s)\rangle=0$. In fact \begin{align}\label{trilinear G^{k,N}}
    \langle G^{k,N}(s), u^N(s)\rangle=b(\sigma_k, u^N(s),u^N(s))=0.
\end{align}
Moreover we have \begin{align}\label{corrector Galerkin}
    2\int_0^t \langle F^N(s),u^N(s)\rangle ds+\sum_{i=1}^N\sum_{k\in K} \lambda_i \int_0^t \langle G^{k,N}(s),e_i\rangle^2 ds=-\sum_{k\in K} b(\sigma_k, u^N(s),(I-P^N)(I-\alpha^2 A)^{-1}P(\sigma_k\cdot\nabla u^N(s))).
\end{align}  In fact, \begin{align*}
    \sum_{i=1}^N \lambda_i  \langle G^{k,N}(s),e_i\rangle^2 \ =b(\sigma_k,u^N(s),\sum_{i=1}^N \lambda_i e_i b(\sigma_k,u^N(s),e_i)).
\end{align*}
It remains to show that \begin{align*}
    &b(\sigma_k,u^N(s),\sum_{i=1}^N \lambda_i e_i b(\sigma_k,u^N(s),e_i))+b(\sigma_k,(I-\alpha^2A)^{-1}P(\sigma_k\cdot \nabla u^N(s))), u^N)=\\ &-b(\sigma_k, u^N(s),(I-P^N)(I-\alpha^2 A)^{-1}P(\sigma_k\cdot\nabla u^N(s))).
\end{align*} 
Thus it is enough to show that  $\langle\sum_{i=1}^N \lambda_i e_i b(\sigma_k,u^N(s),e_i),v\rangle_V=\langle (I-\alpha^2A)^{-1}P(\sigma_k\cdot \nabla u^N(s))),v\rangle_V$ for all $v\in V_N$, where $V_N=\operatorname{span}\{e_i\}_{i=1}^N$. The last claim is true, in fact \begin{align*}
    & \langle(I-\alpha^2A)^{-1}P(\sigma_k\cdot \nabla u^N(s))), v \rangle_V =b(\sigma_k, u^N(s),v)\\ & b(\sigma_k,  u^N(s),\sum_{i=1}^N \lambda_i e_i\langle e_i,v\rangle_v)=\langle\sum_{i=1}^N \lambda_i e_i b(\sigma_k,u^N(s),e_i),v\rangle_V. 
\end{align*}
Therefore, combining equation \eqref{Ito Norm V galerkin} and equations \eqref{trilinear G^{k,N}}, \eqref{corrector Galerkin} we obtain
\begin{align*}
    \lVert u^N(t)\rVert_V^2+2\nu\int_0^t \lVert \nabla u^N(s)\rVert_{L^2}^2 ds & =\lVert u^N_0\rVert_V^2-\sum_{k\in K} \int_0^t  b(\sigma_k,u^N(s),(I-P^N)(I-\alpha^2A)^{-1}P(\sigma_k\cdot \nabla u^N(s)))ds\\ & \leq \lVert  u_0\rVert_V^2+\sum_{k\in K}\lVert \sigma_k\rVert_{L^{\infty}}\int_0^t \lVert \nabla u^N(s)\rVert_{L^2} \lVert(I-P^N)(I-\alpha^2A)^{-1}P(\sigma_k\cdot \nabla u^N(s)) \rVert ds\\ & \leq \lVert  u_0\rVert_V^2+\sum_{k\in K}\lVert \sigma_k\rVert_{L^{\infty}}\int_0^t \lVert u^N(s)\rVert_V \lVert(I-\alpha^2A)^{-1}P(\sigma_k\cdot \nabla u^N(s)) \rVert_Vds\\ & \leq \lVert  u_0\rVert_V^2+\sum_{k\in K}\lVert \sigma_k\rVert_{L^{\infty}}\int_0^t \lVert u^N(s)\rVert_V \lVert P(\sigma_k\cdot \nabla u^N(s)) \rVert ds\\ & \leq \lVert  u_0\rVert_V^2+\frac{1}{\alpha}\sum_{k\in K}\lVert \sigma_k\rVert_{L^{\infty}}^2\int_0^t \lVert u^N(s)\rVert_V ^2ds.
\end{align*}
Thus, by Gr\"onwall,\begin{align}\label{grownall V 2 galerkin}
    \operatorname{sup}_{t\in [0,T]} \lVert u^N(t)\rVert_V^2\leq C_{\alpha,\{\sigma_k\}_{k\in K}} \lVert u_0\rVert_V^2.
\end{align} Taking the expected value of equation \eqref{grownall V 2 galerkin} we get the thesis for $p\leq 2$. If $p>2$, raising to the power $p/2$ both sides of equation \eqref{grownall V 2 galerkin} the thesis follows easily.
\end{proof}
\begin{Lemma}\label{energy in W}
Assuming Hypothesis \ref{hypothesis well-posedness}, the following relation holds:
\begin{align}\label{norm W Galerkin}
    \mathbb{E}\left[\operatorname{sup}_{t\in [0,T]}\lVert u^N(t)\rVert_W^p\right]\leq C_{p,\nu, \alpha,\{\sigma_k\}_{k\in K}},\ \ \forall p\geq 1
\end{align}
where $C_{p,\nu, \alpha,\{\sigma_k\}_{k\in K}}$ is a constant independent from $N$.
\end{Lemma}

\begin{proof}
This proof is similar to Lemma 2.4-2.5 of \cite{razafimandimby2010weak}. We will need some changes due to the poor regularity of the coefficients $F$ and $G^k$. In the part where we will not need any changes, we will refer to the equations in \cite{razafimandimby2010weak}.
Let
\begin{align*}
    \tau_M^N=\inf\{t:\ \lVert u^N(t)\rVert_V+\lVert u^N(t)\rVert_*\geq M \}\wedge T
\end{align*}
and $\Tilde{G}^{k,N}=(I-\alpha^2A)^{-1}G^{k,N}$ the solution of Stokes problem defined in Lemma \ref{Stokes}. From the regularity of the eigenvectors $e_i$, $G^{k,N}\in H^1$, thus $\Tilde{G}^{k,N}\in W$ and by equations \eqref{result inner product eigenfucntions} and \eqref{behave tilde G^k} the following relations hold true \begin{align}
    \langle \Tilde{G}^{k,N},e_i\rangle_W&=\lambda_i\langle G^{k,N},e_i\rangle \label{rel 1 tilde G^kn}\\ 
    \lVert \Tilde{G}^{k,N}\rVert_W&\leq C\lVert \sigma_k\rVert_{W^{1,\infty}}\lVert u^N\rVert_{H^2}.\label{rel 2 tilde G^kn}
\end{align}
Let us call $$\phi^N=-\nu \Delta u^N+\operatorname{curl}(u^N-\alpha^2\Delta u^N)\times u^N-F^N.$$
From the regularity of the $e_i$, we have that $\phi^N\in H^1$. Thus we can find a $v^N\in W$ such that  $v^N=(I-\alpha^2A)^{-1}\phi^N$. We rewrite shortly the weak formulation satisfied by $u^N$ \begin{align*}
    d\langle u^N, e_i\rangle_V+\langle \phi^N,e_i\rangle dt=d\langle u^N, e_i\rangle_V+\langle v^N,e_i\rangle_V dt=\sum_{k\in K}\langle G^{k,N},e_i\rangle dW^k_t
\end{align*}
Multiplying each equation by $\lambda_i$ we get
\begin{align*}
    d\langle u^N, e_i\rangle_W+\langle v^N,e_i\rangle_W dt=\sum_{k\in K}\langle \Tilde{G}^{k,N},e_i\rangle_W dW^k_t.
\end{align*}
Now we apply the It\^o's formula to $\sum_{i=1}^N \langle u^N, e_i\rangle_W^2$ and we obtain
\begin{align*}
    &d (\lVert u^N\rVert_V^2+\lVert u^N\rVert_*^2)+2\left(\langle v^N,u^N\rangle_V+\langle \operatorname{curl}(u^N-\alpha^2\Delta u^N),\operatorname{curl}(v^N-\alpha^2\Delta v^N)\rangle_{L^2}\right)dt=\\& 2\sum_{k\in K} \langle \operatorname{curl}(\Tilde{G}^{k,N}-\alpha^2\Delta \Tilde{G}^{k,N} ),\operatorname{curl}(u^N-\alpha^2\Delta u^N )\rangle_{L^2} dW^k_t\\ & +\sum_{k\in K}\sum_{i=1}^N \lambda_i^2\langle \Tilde{G}^{k,N},e_i \rangle_V^2\ dt+2\sum_{k\in K} \langle \Tilde{G}^{k,N}, u^N\rangle_V dW^k_t.
\end{align*}
Exploiting the definition of $v^N,\ \Tilde{G}^{k,N}$, equation \eqref{trilinear G^{k,N}} and the classical fact that $\operatorname{curl}\nabla=0$ we get
\begin{align}\label{ito formula norm * galerkin step 1}
    &d (\lVert u^N\rVert_V^2+\lVert u^N\rVert_*^2)+2\left(\langle \phi^N,u^N\rangle+\langle \operatorname{curl}(\phi^N),\operatorname{curl}(u^N-\alpha^2\Delta u^N)\rangle_{L^2}\right)dt=\notag\\& 2\sum_{k\in K} \langle \operatorname{curl}({G}^{k,N} ),\operatorname{curl}(u^N-\alpha^2\Delta u^N )\rangle_{L^2} dW^k_t +\sum_{k\in K}\sum_{i=1}^N \lambda_i^2\langle {G}^{k,N},e_i \rangle^2\ dt.
\end{align}
From Lemma \ref{energy in V} we already know that \begin{align*}
    d \lVert u^N\rVert_V^2=-2\nu \lVert \nabla u^N\rVert_{L^2}^2 dt-\sum_{k\in K} b(\sigma_k, u^N,(I-P^N)(I-\alpha^2 A)^{-1}P(\sigma_k\cdot\nabla u^N))dt.
\end{align*} Substituting this relation in the It\^o's formula \eqref{ito formula norm * galerkin step 1} we get
\begin{align}\label{ito formula norm * galerkin step 2}
    &d (\lVert u^N\rVert_*^2)+2\left(\langle \operatorname{curl}(\phi^N),\operatorname{curl}(u^N-\alpha^2\Delta u^N)\rangle_{L^2}\right)dt= \notag\\& 2\sum_{k\in K} \langle \operatorname{curl}({G}^{k,N} ),\operatorname{curl}(u^N-\alpha^2\Delta u^N )\rangle_{L^2} dW^k_t +\sum_{k\in K}\sum_{i=1}^N (\lambda_i+\lambda_i^2)\langle {G}^{k,N},e_i \rangle^2\ dt.
\end{align}
Analogously to equation (4.48) in \cite{razafimandimby2010weak}, the relation below holds true\begin{align*}
    \langle \operatorname{curl}\phi^N,\operatorname{curl}(u^N-\alpha^2\Delta u^N)\rangle_{L^2}=\frac{\nu}{\alpha^2}\lVert u^N\rVert_*^2-\frac{\nu}{\alpha^2}\langle\operatorname{curl}u^N+\frac{\alpha^2}{\nu}\operatorname{curl}F^N,\operatorname{curl}(u^N-\alpha^2\Delta u^N) \rangle_{L^2}.
\end{align*}
Using this relation in the It\^o's formula \eqref{ito formula norm * galerkin step 2} and integrating between $0$ and $t\leq \tau_M^N$ we get
\begin{align}\label{ito formula norm * galerkin step 2.5}
    &\lVert u^N(t)\rVert_*^2+\frac{2\nu}{\alpha^2}\int_0^t \lVert u^N(s)\rVert_*^2 - \sum_{k\in K}\sum_{i=1}^N(\lambda_i+\lambda_i^2)\langle {G}^{k,N}(s),e_i \rangle^2ds\notag\\& =\lVert u^N_0\rVert_*^2+\int_0^t \frac{2\nu}{\alpha^2}\langle\operatorname{curl}u^N(s)+\frac{\alpha^2}{\nu}\operatorname{curl}F^N(s),\operatorname{curl}(u^N(s)-\alpha^2\Delta u^N(s)) \rangle_{L^2} ds\notag\\ &+2\sum_{k\in K} \int_0^t\langle \operatorname{curl}({G}^{k,N}(s) ),\operatorname{curl}(u^N(s)-\alpha^2\Delta u^N(s) )\rangle_{L^2} dW^k_s \notag\\ &\leq \lVert u^N_0\rVert_*^2+\int_0^t \frac{2\nu}{\alpha^2}\lVert \operatorname{curl}u^N(s)\rVert_{L^2}\lVert u^N(s)\rVert_*ds+2\int_0^t \lVert \operatorname{curl}F^N(s)\rVert_{L^2}\lVert u^N(s)\rVert_*ds\notag\\& +2\left\lvert \sum_{k\in K} \int_0^t\langle \operatorname{curl}({G}^{k,N}(s) ),\operatorname{curl}(u^N(s)-\alpha^2\Delta u^N(s) )\rangle_{L^2} dW^k_s \right\rvert.
\end{align}
Taking the supremum between $0$ and $r\wedge \tau_M^N$ in relation \eqref{ito formula norm * galerkin step 2.5} and, then, the expected value we get
\begin{align}\label{ito formula norm * galerkin step 3}
    &\mathbb{E}\left[\operatorname{sup}_{t\leq r\wedge \tau_M^N}\lVert u^N(t)\rVert_*^2\right]+\frac{2\nu}{\alpha^2}\mathbb{E}\left[\int_0^{r\wedge \tau_M^N} \lVert u^N(s)\rVert_*^2ds\right] \notag\\& \leq2 \mathbb{E}\left[\lVert u^N_0\rVert_*^2\right]+\mathbb{E}\left[\int_0^{r\wedge \tau_M^N} \frac{4\nu}{\alpha^2}\lVert \operatorname{curl}u^N(s)\rVert_{L^2}\lVert u^N(s)\rVert_*ds\right]+4\mathbb{E}\left[\int_0^{r\wedge \tau_M^N} \lVert \operatorname{curl}F^N(s)\rVert_{L^2}\lVert u^N(s)\rVert_*ds\right]\notag\\& +4\mathbb{E}\left[\left\lvert \sum_{k\in K} \int_0^{r\wedge \tau_M^N}\langle \operatorname{curl}({G}^{k,N}(s) ),\operatorname{curl}(u^N(s)-\alpha^2\Delta u^N(s) )\rangle_{L^2} dW^k_s\right\rvert\right]\notag\\ & \leq 2\mathbb{E}\left[\lVert u^N_0\rVert_*^2\right]+4\mathbb{E}\left[\operatorname{sup}_{t\leq r\wedge \tau_M^N}\left\lvert \sum_{k\in K} \int_0^{t}\langle \operatorname{curl}({G}^{k,N}(s) ),\operatorname{curl}(u^N(s)-\alpha^2\Delta u^N(s) )\rangle_{L^2} dW^k_s\right\rvert\right]\notag\\ &+\mathbb{E}\left[\int_0^{r\wedge \tau_M^N} \left(\frac{2\nu\epsilon_1}{\alpha^2}+2\epsilon_2\right)\lVert u^N(s)\rVert_*^2 ds\right]+\mathbb{E}\left[\int_0^{r\wedge \tau_M^N}\frac{2\nu}{\alpha^2\epsilon_1}\lVert \operatorname{curl}u^N(s)\rVert_{L^2}^2ds\right]\notag\\ &+ \mathbb{E}\left[\int_0^{r\wedge \tau_M^N}\frac{2}{\epsilon_2}\lVert \operatorname{curl} F^N(s)\rVert_{L^2}^2 ds \right]+2\sum_{k\in K}\sum_{i=1}^N(\lambda_i+\lambda_i^2)\mathbb{E}\left[\int_0^{r\wedge \tau_M^N}\langle {G}^{k,N}(s),e_i \rangle^2ds\right].
\end{align}
Choosing $\epsilon_1=\frac{1}{4}$ and $\epsilon_2=\frac{\nu}{4\alpha^2}$ we arrive at
\begin{align}\label{energy star}
    &\mathbb{E}\left[\operatorname{sup}_{t\leq r\wedge \tau_M^N}\lVert u^N(t)\rVert_*^2\right]+\frac{\nu}{\alpha^2}\mathbb{E}\left[\int_0^{r\wedge \tau_M^N} \lVert u^N(s)\rVert_*^2ds\right] \notag\\ & \leq 2\mathbb{E}\left[\lVert u^N_0\rVert_*^2\right]+4\mathbb{E}\left[\operatorname{sup}_{t\leq r\wedge \tau_M^N}\left\lvert \sum_{k\in K} \int_0^{t}\langle \operatorname{curl}({G}^{k,N}(s) ),\operatorname{curl}(u^N(s)-\alpha^2\Delta u^N(s) )\rangle_{L^2} dW^k_s\right\rvert\right]\notag\\ &+\mathbb{E}\left[\int_0^{r\wedge \tau_M^N}\frac{8\nu}{\alpha^2}\lVert \operatorname{curl}u^N(s)\rVert_{L^2}^2ds\right]+ \mathbb{E}\left[\int_0^{r\wedge \tau_M^N}\frac{8\alpha^2}{\nu}\lVert \operatorname{curl} F^N(s)\rVert_{L^2}^2 ds \right]+2\sum_{k\in K}\sum_{i=1}^N(\lambda_i+\lambda_i^2)\mathbb{E}\left[\int_0^{r\wedge \tau_M^N}\langle {G}^{k,N}(s),e_i \rangle^2ds\right]
\end{align}
From equations \eqref{corrector Galerkin} and \eqref{G^k H H^2} we know that
\begin{align}\label{estimate 1}
    \sum_{k\in K}\sum_{i=1}^N\lambda_i\mathbb{E}\left[\int_0^{r\wedge \tau_M^N}\langle {G}^{k,N}(s),e_i \rangle^2ds\right]=&\sum_{k\in K}\mathbb{E}\left[\int_0^{r\wedge \tau_M^N} b(\sigma_k,u^N(s),P^N(I-\alpha^2A)^{-1}P(\sigma_k\cdot\nabla u^N(s)))ds\right]\notag\\ & \leq \sum_{k\in K} \lVert \sigma_k\rVert_{L^{\infty}} \mathbb{E}\left[\int_0^{r\wedge \tau_M^N}\lVert \nabla u^N(s)\rVert_{L^2} \lVert (I-\alpha^2A)^{-1}P(\sigma_k\cdot \nabla u^N(s))\rVert_V ds\right]\notag\\ & \leq \sum_{k\in K} \lVert \sigma_k\rVert_{L^{\infty}} \mathbb{E}\left[\int_0^{r\wedge \tau_M^N}\lVert \nabla u^N(s)\rVert_{L^2} \lVert (I-\alpha^2A)^{-1/2}P(\sigma_k\cdot \nabla u^N(s))\rVert ds\right]\notag\\ & \leq \sum_{k\in K} \lVert \sigma_k\rVert_{L^{\infty}}^2 \mathbb{E}\left[\int_0^{r\wedge \tau_M^N}\lVert \nabla u^N(s)\rVert_{L^2}^2 ds \right].
\end{align}
Thanks to equations \eqref{rel 1 tilde G^kn}, \eqref{rel 2 tilde G^kn}, the interpolation estimate \eqref{interpolation estimate} and relation \eqref{scaling H^3} we have
\begin{align}\label{estimate 2}
\sum_{k\in K}\sum_{i=1}^N\lambda_i^2\mathbb{E}\left[\int_0^{r\wedge \tau_M^N}\langle {G}^{k,N}(s),e_i \rangle^2ds\right]=&  \sum_{k\in K}\sum_{i=1}^N\mathbb{E}\left[\int_0^{r\wedge \tau_M^N}\langle \Tilde{G}^{k,N}(s),e_i \rangle_W^2ds\right]\notag\\ & \leq \sum_{k\in K}  \mathbb{E}\left[\int_0^{r\wedge \tau_M^N}\lVert \Tilde{G}^{k,N}(s)\rVert_W^2ds\right]\notag\\ & \leq C\sum_{k\in K}\lVert \sigma_k\rVert_{W^{1,\infty}}^2 \mathbb{E}\left[\int_0^{r\wedge \tau_M^N}\lVert u^{N}(s)\rVert_{H^2}^2ds\right]\notag\\ & \leq C\sum_{k\in K}\lVert \sigma_k\rVert_{W^{1,\infty}}^2 \mathbb{E}\left[\int_0^{r\wedge \tau_M^N}\lVert \nabla u^{N}(s)\rVert_{L^2}\lVert u^{N}(s)\rVert_{H^3}ds\right]\notag\\ & \leq C\sum_{k\in K}\lVert \sigma_k\rVert_{W^{1,\infty}}^2 \mathbb{E}\left[\int_0^{r\wedge \tau_M^N}\lVert \nabla u^{N}(s)\rVert_{L^2}\left(\frac{\alpha^2+1}{\alpha^2}\lVert \nabla u^{N}(s)\rVert_{L^2}+\frac{1}{\alpha^2}\lVert u^N(s)\rVert_*\right)ds\right]\notag\\ & \leq \sum_{k\in K}\lVert \sigma_k\rVert_{W^{1,\infty}}^2 \mathbb{E}\left[\int_0^{r\wedge \tau_M^N} C\frac{\alpha^2+1+\frac{1}{\epsilon}}{\alpha^2}\lVert \nabla u^{N}(s)\rVert_{L^2}^2+\frac{\epsilon}{\alpha^2}\lVert u^N(s)\rVert_*^2 ds\right].
\end{align}
Thanks to Burkholder-Davis-Gundy inequality, equation \eqref{behave G^k H^2 H^1}, the interpolation inequality \eqref{interpolation estimate} and relation \eqref{scaling H^3} we get
\begin{align}\label{estimate 3}
    &4\mathbb{E}\left[\operatorname{sup}_{t\leq r\wedge \tau_M^N}\left\lvert \sum_{k\in K} \int_0^{t}\langle \operatorname{curl}({G}^{k,N}(s) ),\operatorname{curl}(u^N(s)-\alpha^2\Delta u^N(s) )\rangle_{L^2} dW^k_s\right\rvert\right]\notag\\ &\leq C \mathbb{E}\left[\left(\sum_{k\in K} \int_0^{r\wedge \tau_M^N}\lVert \operatorname{curl}G^{k,N}(s)\rVert_{L^2}^2\lVert u^N(s)\rVert_*^2ds\right)^{1/2}\right]\notag\\ &\leq C \mathbb{E}\left[\operatorname{sup}_{t\leq r\wedge \tau_M^N}\lVert u^N(t)(s)\rVert_*\left(\sum_{k\in K} \int_0^{r\wedge \tau_M^N}\lVert \operatorname{curl}G^{k,N}(s)\rVert_{L^2}^2 ds\right)^{1/2}\right]\notag\\  &\leq \frac{1}{2}\mathbb{E}\left[\operatorname{sup}_{t\leq r\wedge \tau_M^N}\lVert u^N(t)\rVert_*^2\right]+C\mathbb{E}\left[\sum_{k\in K} \int_0^{r\wedge \tau_M^N}\lVert \operatorname{curl}G^{k,N}(s)\rVert_{L^2}^2 ds \right] \notag\\ & \leq \frac{1}{2}\mathbb{E}\left[\operatorname{sup}_{t\leq r\wedge \tau_M^N}\lVert u^N(t)\rVert_*^2\right]+C\sum_{k\in K}\lVert \sigma_k\rVert_{W^{1,\infty}}^2\mathbb{E}\left[\int_0^{r\wedge \tau_M^N} \lVert  u^{N}(s)\rVert_{H^2}^2 ds\right]\notag\\  &\leq\frac{1}{2}\mathbb{E}\left[\operatorname{sup}_{t\leq r\wedge \tau_M^N}\lVert u^N(t)\rVert_*^2\right]+\sum_{k\in K}\lVert \sigma_k\rVert_{W^{1,\infty}}^2 \mathbb{E}\left[\int_0^{r\wedge \tau_M^N} C\frac{\alpha^2+1+\frac{1}{\epsilon}}{\alpha^2}\lVert \nabla u^{N}(s)\rVert_{L^2}^2+\frac{\epsilon}{\alpha^2}\lVert u^N(s)\rVert_*^2 ds\right].
\end{align}

Lastly, thanks to equation \eqref{behave F V H1} we have
\begin{align}\label{estimate 4}
    \mathbb{E}\left[\int_0^{r\wedge \tau_M^N}\frac{8\alpha^2}{\nu}\lVert \operatorname{curl} F^N(s)\rVert_{L^2}^2 ds \right]& \leq \frac{C}{\nu\alpha^2}\left(\sum_{k\in K}\lVert \sigma_k\rVert_{L^{\infty}}\lVert \sigma_k\rVert_{W^{1,\infty}}\right)^2\mathbb{E}\left[\int_0^{r\wedge \tau_M^N}\lVert \nabla u^N(s)\rVert_{L^2}^2 ds\right].
\end{align}
Combining estimates \eqref{estimate 1},\eqref{estimate 2},\eqref{estimate 3},\eqref{estimate 4} above we obtain
\begin{align}\label{final norm W 2}
    &\mathbb{E}\left[\operatorname{sup}_{t\leq r\wedge \tau_M^N}\lVert u^N(t)\rVert_*^2\right]+\frac{\nu-\epsilon\sum_{k\in K}\lVert \sigma_k\rVert_{W^{1,\infty}}^2}{\alpha^2}\mathbb{E}\left[\int_0^{r\wedge \tau_M^N} \lVert u^N(s)\rVert_*^2ds\right] \notag\\ & \leq C\left(\mathbb{E}\left[\lVert u^N_0\rVert_*^2\right]+\frac{\alpha^2+1+\frac{1}{\epsilon}}{\alpha^2}\sum_{k\in K}\lVert \sigma_k\rVert_{W^{1,\infty}}^2 \mathbb{E}\left[\int_0^{r\wedge \tau_M^N} \lVert \nabla u^{N}(s)\rVert_{L^2}^2\right]\right)\notag\\ & +\frac{C}{\nu\alpha^2}\left(\sum_{k\in K}\lVert \sigma_k\rVert_{L^{\infty}}\lVert \sigma_k\rVert_{W^{1,\infty}}\right)^2\mathbb{E}\left[\int_0^{r\wedge \tau_M^N}\lVert \nabla u^N(s)\rVert_{L^2}^2 ds\right].
\end{align}

Therefore, choosing $\epsilon$ small enough, by equation \eqref{grownall V 2 galerkin} we have
\begin{align*}
    \mathbb{E}\left[\operatorname{sup}_{t\leq r\wedge \tau_M^N}\lVert u^N(t)\rVert_*^2\right]+\frac{\nu}{2\alpha^2}\mathbb{E}\left[\int_0^{r\wedge \tau_M^N} \lVert u^N(s)\rVert_*^2ds\right]\leq C_{\nu, \alpha,\{\sigma_k\}_{k\in K}} \ \text{indipendent from }M,\ N.
\end{align*}
Last inequality proves the Lemma for $p=2$, letting $M$ to $+\infty$ thanks to monotone convergence Theorem. Now we consider $p\geq 4$ and we restart from equation (4.79) in \cite{razafimandimby2010weak}.
\begin{align*}
    \lVert u^N(t)\rVert_*^p &\leq C\lVert u^N_0\rVert_*^p+C \bigg(\int_0^t \lVert u^N(s)\rVert_*^{p/2-2} \\ &\times \bigg(2\langle \operatorname{curl}F^N(s),\operatorname{curl}(u^N(s)-\alpha^2\Delta u^N(s))\rangle_{L^2}\\ & +\frac{1}{2}\sum_{k\in K}\sum_{i=1}^N(\lambda_i+\lambda_i^2) \langle G^{k,N}(s),e_i\rangle^2\\ &+\frac{2\nu}{\alpha^2}\langle \operatorname{curl}u^N(s),\operatorname{curl}(u^N(s)-\alpha^2\Delta u^N(s))\rangle_{L^2} -\frac{2\nu}{\alpha^2}\lVert u^N(s)\rVert_*^2\\ & +\frac{p-4}{p}\sum_{k\in K}\frac{\langle \operatorname{curl}G^{k,N}(s),\operatorname{curl}(u^N(s)-\alpha^2\Delta u^N(s))\rangle_{L^2}^2}{\lVert u^N(s)\rVert_*^2}\bigg)ds\bigg)^2\\ & + \left(\int_0^t\lVert u^N(s)\rVert_*^{p/2-2}\sum_{k\in K}\langle \operatorname{curl}G^{k,N}(s),\operatorname{curl}(u^N(s)-\alpha^2\Delta u^N(s))\rangle_{L^2} dW^k_s\right)^2.
\end{align*}
Let us consider all the terms, one by one. Arguing as before we have
\begin{align*}
\sum_{k\in K}\sum_{i=1}^N(\lambda_i+\lambda_i^2)\langle {G}^{k,N}(s),e_i \rangle^2\leq C_{\epsilon,\alpha,\{\sigma_k\}_{k\in K}}\lVert u^N(s)\rVert_V^2+\epsilon \lVert u^N(s)\rVert_*^2,
\end{align*}
\begin{align*}
    \lvert \langle \operatorname{curl}u^N(s),\operatorname{curl}(u^N(s)-\alpha^2\Delta u^N(s))\rangle_{L^2}\rvert & \leq C_{\alpha}(1+\lVert u^N(s)\rVert_V)(1+\lVert u^N(s)\rVert_W),
\end{align*}
\begin{align*}
    \lvert\langle \operatorname{curl}F^N(s),\operatorname{curl}(u^N(s)-\alpha^2\Delta u^N(s))\rangle_{L^2}\leq  C_{\alpha,\{\sigma_k\}_{k\in K}}(1+\lVert u^N(s)\rVert_V)(1+\lVert u^N(s)\rVert_W),
\end{align*}
\begin{align*}
    \left\lvert \frac{\langle \operatorname{curl}G^{k,N}(s),\operatorname{curl}(u^N(s)-\alpha^2\Delta u^N(s))\rangle_{L^2}^2}{\lVert u^N(s)\rVert_*^2}\right\rvert & \leq \lVert \operatorname{curl} G^{k,N}(s)\rVert_{L^2}^2\\ & \leq C_{\epsilon, \alpha,\{\sigma_k\}_{k\in K}}(1+\lVert u^N(s)\rVert_V)^2+\epsilon\rVert u^N(s)\rVert_W^2.
\end{align*}
Exploiting the relations above and the continuous embedding $W\hookrightarrow V$ we get
\begin{align*}
    \lVert u^N(t)\rVert_*^p &\leq C\lVert u^N_0\rVert_*^p+ C_{\epsilon,\nu,\alpha,\{\sigma_k\}_{k\in K}} \left(\int_0^t \lVert u^N(s)\rVert_*^{p/2-2} (1+\lVert u^N(s)\rVert_W)^2ds\right)^2\\ &+\left(\int_0^t\lVert u^N(s)\rVert_*^{p/2-2}\sum_{k\in K}\langle \operatorname{curl}G^{k,N}(s),\operatorname{curl}(u^N(s)-\alpha^2\Delta u^N(s))\rangle_{L^2} dW^k_s\right)^2.
\end{align*}
Thus taking the supremum in time for $t\leq r$ and the expected value of this we get the thesis via Grownall's Lemma arguing exactly as in the proof of Lemma 4.3 in \cite{razafimandimby2010weak} and exploiting previous estimate \eqref{estimate 3} on $\langle \operatorname{curl}G^{k,N}(s),\operatorname{curl}(u^N(s)-\alpha^2\Delta u^N(s))\rangle_{L^2}$.
\end{proof}
\begin{remark}\label{changes Lemma energy W} 
In case of $\nu=0 $, arguing as above we get 
\begin{align}
    \mathbb{E}\left[\operatorname{sup}_{t\leq r\wedge \tau_M^N}\lVert u^N(t)\rVert_*^2\right]& \leq 2\mathbb{E}\left[\lVert u^N_0\rVert_*^2\right]+\mathbb{E}\left[\int_0^{r\wedge \tau_M^N} \lVert u^N(s)\rVert_*^2 ds\right]\notag\\ & +4\mathbb{E}\left[\operatorname{sup}_{t\leq r\wedge \tau_M^N}\left\lvert \sum_{k\in K} \int_0^{t}\langle \operatorname{curl}({G}^{k,N}(s) ),\operatorname{curl}(u^N(s)-\alpha^2\Delta u^N(s) )\rangle_{L^2} dW^k_s\right\rvert\right]\notag\\ &+ 4\mathbb{E}\left[\int_0^{r\wedge \tau_M^N}\lVert \operatorname{curl} F^N(s)\rVert_{L^2}^2 ds \right]+2\sum_{k\in K}\sum_{i=1}^N(\lambda_i+\lambda_i^2)\mathbb{E}\left[\int_0^{r\wedge \tau_M^N}\langle {G}^{k,N}(s),e_i \rangle^2ds\right]\label{ito formula norm * galerkin step 3 remark}.
\end{align}
Therefore, thanks to Lemma \ref{energy in V} and estimates \eqref{estimate 1},\eqref{estimate 2},\eqref{estimate 3},\eqref{estimate 4} we obtain 

\begin{align}\label{final norm W 2 remark}
    \mathbb{E}\left[\operatorname{sup}_{t\leq r\wedge \tau_M^N}\lVert u^N(t)\rVert_*^2\right]& \leq   \frac{\sum_{k\in K}\lVert \sigma_k\rVert_{W^{1,\infty}}^2}{\alpha^2}\mathbb{E}\left[\int_0^{r\wedge \tau_M^N} \lVert u^N(s)\rVert_*^2ds\right]\notag\\ &+ C\left(\mathbb{E}\left[\lVert u^N_0\rVert_*^2\right]+\frac{\alpha^2+1}{\alpha^2}\sum_{k\in K}\lVert \sigma_k\rVert_{W^{1,\infty}}^2 \mathbb{E}\left[\int_0^{r\wedge \tau_M^N} \lVert \nabla u^{N}(s)\rVert_{L^2}^2\right]\right)\notag\\ & +\frac{C}{\alpha^4}\left(\sum_{k\in K}\lVert \sigma_k\rVert_{L^{\infty}}\lVert \sigma_k\rVert_{W^{1,\infty}}\right)^2\mathbb{E}\left[\int_0^{r\wedge \tau_M^N}\lVert \nabla u^N(s)\rVert_{L^2}^2 ds\right]\notag\\ & \leq C_{\alpha,\{\sigma_k\}_{k\in K}}+C_{\alpha,\{\sigma_k\}_{k\in K}}\mathbb{E}\left[\int_0^{r\wedge \tau_M^N} \lVert u^N(s)\rVert_*^2ds\right]\notag \\ & \leq C_{\alpha,\{\sigma_k\}_{k\in K}}+C_{\alpha,\{\sigma_k\}_{k\in K}}\int_0^r \mathbb{E}\left[ \lVert u^N(s)\rVert_*^21_{[0,\tau_M^N]}(s)ds\right] ds.
\end{align}
Since $\mathbb{E}\left[\operatorname{sup}_{t\leq r\wedge \tau_M^N}\lVert u^N(t)\rVert_*^2\right]=\mathbb{E}\left[\operatorname{sup}_{t\leq r}\lVert u^N(t)\rVert_*^21_{[0,\tau_M^N]}(t)\right],$
by Gr\"onwall's Lemma 
 \begin{align*}
     \mathbb{E}\left[\operatorname{sup}_{t\in [0,T\wedge\tau_M^N]}\lVert u^N(t)\rVert_W^2\right]\leq C_{\alpha,\{\sigma_k\}_{k\in K}} \textit{ indipendent from }M,\ N.
 \end{align*}
 Last inequality proves the Lemma for $p=2$, letting $M$ to $\infty$ thanks to monotone convergence Theorem. The case $p\geq 4$ can be treated as in the case $\nu>0$, therefore we do not add other details.
\end{remark}
Let us now introduce the operator $\hat{A}=(I-\alpha^2A)^{-1}A$. By Lemma \ref{nonlinearity} and Lemma \ref{Stokes} the weak formulation satisfied by the Galerkin approximations can be rewritten as \begin{align*}
    \langle u^N(t), e_i\rangle_V-  \langle u^N_0, e_i\rangle_V&=\nu \int_0^t \langle \hat{A} u^N(s),  e_i\rangle_V ds- \int_0^t \langle \hat{B}(u^N(s),u^N(s)),e_i\rangle_{W^*,W}ds+\int_0^t \langle F^N(s),e_i \rangle ds \\ &+ \sum_{k\in K} \int_0^t \langle G^{k,N}(s),e_i \rangle  dW^k_s \ \ \mathbb{P}-a.s.
\end{align*}
Thanks to relations \eqref{norm V galerkin},\eqref{norm W Galerkin} and the continuity of $B,\ F$ and $G^k$, we know that exists a subsequence of the Galerkin approximations, which we will denote again by $u^N$ just for simplicity, and processes $u$ and $\hat{B}^*$ such that \begin{align}\label{convergences}\begin{cases}
    u^N\stackrel{*}{\rightharpoonup}u  \text{ in } L^p(\Omega,\mathcal{F},\mathbb{P};L^{\infty}(0,T;W)),\ \ p\geq 2\\
    u^N{\rightharpoonup}u  \text{ in } L^p(\Omega,\mathcal{F},\mathbb{P};L^{q}(0,T;V)),\ \ \ p,\ q\geq 2\\ 
    \hat{A}u^N{\rightharpoonup}\hat{A}u  \text{ in } L^2(\Omega,\mathcal{F},\mathbb{P};L^{2}(0,T;V))\\
    \hat{B}(u^N,u^N)\rightharpoonup \hat{B}^* \text{ in } L^2(\Omega,\mathcal{F},\mathbb{P};L^{2}(0,T;W^*))\\
    F(u^N)\rightharpoonup F(u)  \text{ in } L^2(\Omega,\mathcal{F},\mathbb{P};L^{2}(0,T;H\cap H^1))\\ 
    G^k(u^N)\rightharpoonup G^k(u)  \text{ in } L^2(\Omega,\mathcal{F},\mathbb{P};L^{2}(0,T;H))
\end{cases}
    \end{align}

Next step will be showing that $\hat{B}^*=\hat{B}(u,u)$. In this way the existence of a solution of equation (\ref{Ito}) will follow. In fact, we know that $\mathbb{P}-a.s.$ for each $i\in \mathbb{N}$,\ for each $t\in [0,T]$
\begin{align*}
    \langle u(t), e_i\rangle_V-  \langle u_0, e_i\rangle_V&=\nu \int_0^t \langle \hat{A} u(s),  e_i\rangle_V ds- \int_0^t \langle \hat{B}^*(s),e_i\rangle_{W^*,W}ds+\int_0^t \langle F(u(s)),e_i \rangle ds \\ &+ \sum_{k\in K} \int_0^t \langle G^{k}(u(s)),e_i \rangle  dW^k_s. \ \
\end{align*}
For what concern the continuity in $V$ we can argue in the following way via It\^o's formula and Kolmogorov continuity Theorem.
From the weak formulation above we get the weak continuity in $V$ of $u$ applying the Kolmogorov continuity Theorem for the SDE satisfied by $\langle u(t), e_i\rangle_V$, applying the It\^o's formula to $\lVert u\rVert_V^2$ we get
\begin{align*}
    d \lVert u\rVert_V^2=-2\nu\lVert\nabla u\rVert_{L^2}^2 dt -\langle \hat{B}^*,u\rangle_{W^*,W}dt. 
\end{align*}
From this, we get the continuity of $\lVert u\rVert_V^2$ thanks to the integrability properties of $u$. Weak continuity and continuity of the norm implies strong continuity, thus we have the strong continuity of $u$ as a process taking values in $V$. Weak continuity of $u$ as a process taking values in $W$ follows from Lemma 1.4, pag. 263 in \cite{temam2001navier}. Alternatively the strong continuity in $V$ of $u$ follows arguing as in \cite{breckner2000galerkin} or \cite{pardoux1975equations}.
\subsection{Existence, Uniqueness and Further Results}
To prove the existence of the solutions of equations (\ref{Ito}) we need the following Lemma. As stated in Section \ref{results}, this way of proceed  has been introduced in \cite{breckner2000galerkin} for Navier-Stokes equations.
\begin{Lemma}\label{crucial lemma}
Let $$\tau_M=\operatorname{inf}\{t\in [0,T]: \lVert u(t)\rVert_V+\lVert u(t)\rVert_*\geq M\}\wedge T$$
then \begin{align*}
    1_{t\leq \tau_M}(u^N-u)\rightarrow 0 \text{\ in\ } L^2(\Omega,\mathcal{F},\mathbb{P};L^2(0,T;V)).
\end{align*}
\end{Lemma}
\begin{proof}
Let $P^N$ be the projection of $W$ on $W^N=\operatorname{span}\{e_1,\dots, e_N\}.$ Thanks to dominated convergence Theorem, 
\begin{align}\label{preliminary observation P^N}
    P^N w\rightarrow w \quad \text{in }L^r(\Omega,L^q(0,T;W)) \quad \text{ if } r,\ q\in [1,+\infty)\text{ and } w\in L^r(\Omega,L^q(0,T;W)).
\end{align}
Consequently we have also convergence in $L^r(\Omega,L^q(0,T;V))$. Moreover, if $w\in W,\ i\leq N,\ \langle P^Nw,e_i\rangle_V=\langle w,e_i\rangle_V$. Let $\hat{F}(u)=(I-\alpha^2 A)^{-1}F(u)$, $\hat{G}^k(u)=(I-\alpha^2 A)^{-1}G^k(u)$. From the weak formulation satisfied by $u$, for each $i\leq N$, we get
\begin{align*}
    \langle P^N u(t), e_i\rangle_V-  \langle u^N_0, e_i\rangle_V&=\nu \int_0^t \langle P^N\hat{A} u(s),  e_i\rangle_V ds- \int_0^t \langle \hat{B}^*(s),e_i\rangle_{W^*,W}ds+\int_0^t \langle \hat{F}(u(s)),e_i \rangle_V ds \\ &+ \sum_{k\in K} \int_0^t \langle \hat{G}^k(u(s)),e_i \rangle  dW^k_s \ \ \mathbb{P}-a.s.
\end{align*}
Exploiting the relation satisfied by $u^N$, we get
\begin{align}\label{preliminary Ito crucial Lemma}
    \langle P^N u(t)-u^N(t), e_i\rangle_V&=\nu \int_0^t \langle P^N\hat{A} (u(s)-u^N(s)),  e_i\rangle_V ds+ \int_0^t \langle \hat{B}(u^N(s),u^N(s))-\hat{B}^*(s),e_i\rangle_{W^*,W}ds\notag\\ &+\int_0^t \langle \hat{F}(u(s))-\hat{F}^N(s),e_i \rangle_V ds + \sum_{k\in K} \int_0^t \langle \hat{G}^k(u(s))-\hat{G}^{k,N}(s),e_i \rangle  dW^k_s \quad \mathbb{P}-a.s.
\end{align}
Thanks to \eqref{preliminary Ito crucial Lemma}, applying the It\^o's formula to $\sigma(t)\lVert P^N u(t)-u^N(t)\rVert_V^2$, where $\sigma(t)=exp(-\eta_1 t-\eta_2\int_0^t \lVert u(s)\rVert_W^2 ds)$, we obtain
\begin{align}\label{energy Lemma}
    &\sigma(t)\lVert P^N u(t)-u^N(t)\rVert_V^2+2\nu \int_0^t \sigma(s)\langle \hat{A}(u(s)-u^N(s)),P^Nu(s)-u^N(s)\rangle_Vds \notag\\ &= \sum_{i=1}^N\sum_{k\in K} \lambda_i\int_0^t \sigma(s) \langle \hat{G}^k(u(s))-\hat{G}^{k,N}(s),e_i\rangle_V^2ds\notag\\ &+2\sum_{k\in K} \int_0^t \sigma(s)\langle \hat{G}^k(u(s))-\hat{G}^{k,N}(s),P^Nu(s)-u^N(s)\rangle_V dW^k_s \notag\\ & -\eta_1\int_0^t \sigma(s)\lVert P^Nu(s)-u^N(s)\rVert_V^2ds-\eta_2\int_0^t \sigma(s)\lVert P^Nu(s)-u^N(s)\rVert_V^2\lVert u(s)\rVert_W^2ds\notag\\ & +2 \int_0^t \sigma(s)\langle \hat{B}(u^N(s),u^N(s))-\hat{B}^*(s),P^Nu(s)-u^N(s)\rangle_{W^*,W} ds \notag\\& +2 \int_0^t \sigma(s)\langle \hat{F}(u(s))-\hat{F}^{N}(s),P^Nu(s)-u^N(s)\rangle_V ds.
\end{align}
Let us analyze the terms in \eqref{energy Lemma} one by one. We will not add details where the computations are analogous to Lemma 3.9 in \cite{razafimandimby2012strong}.
\begin{align*}
\langle \hat{A}(u(s)-u^N(s)),P^Nu(s)-u^N(s)\rangle_V=\langle \hat{A}(u(s)-u^N(s)),u(s)-u^N(s)\rangle_V -\langle \hat{A}(u(s)-u^N(s)),u(s)-P^Nu(s)\rangle_V,             
\end{align*}
\begin{align*}
\frac{1}{C_p^2+\alpha^2}\lVert P^Nu(s)-u^N(s)\rVert_V^2\leq \langle \hat{A}(u(s)-u^N(s)),u(s)-u^N(s) \rangle_V,
\end{align*}
\begin{align*}
&\langle \hat{F}(u(s))-\hat{F}^{N}(s),P^Nu(s)-u^N(s)\rangle_V\\& \stackrel{\text{equation } \eqref{behave F V H}}{\leq} C^F_{\alpha,\{\sigma_k\}_{k\in K}}\lVert u(s)-u^N(s)\rVert_V^2\leq 2C^F_{\alpha,\{\sigma_k\}_{k\in K}}\lVert u^N(s)-P^N u(s)\rVert_V^2+2C^F_{\alpha,\{\sigma_k\}_{k\in K}}\lVert u(s)-P^N u(s)\rVert_V^2,
\end{align*}
\begin{align*}
\sum_{i=1}^N\sum_{k\in K} \lambda_i  \langle \hat{G}^k(u(s))-\hat{G}^{k,N}(s),e_i\rangle_V^2 &=\sum_{k\in K}\lVert P^N (\hat{G}^k(u(s))-\hat{G}^{k,N}(s))\rVert_V^2\\ &\stackrel{\text{equation }\eqref{behave G^k V H}}{\leq} 2C^G_{\alpha,\{\sigma_k\}_{k\in K}}\lVert u^N(s)-P^N u(s)\rVert_V^2+2C^G_{\alpha,\{\sigma_k\}_{k\in K}}\lVert u(s)-P^N u(s)\rVert_V^2,
\end{align*}

\begin{align*}
    2\langle \hat{B}(u^N(s),u^N(s))-\hat{B}^*(s),P^Nu(s)-u^N(s)\rangle_{W^*,W}&\leq C_B^2\lVert  P^Nu(s)-u^N(s)\rVert_V^2\lVert P^N(s) u(s)\rVert_W^2+\lVert  P^Nu(s)-u^N(s)\rVert_V^2\\ & +2 \langle \hat{B}(P^Nu(s),P^Nu(s))-\hat{B}^*(s),P^Nu(s)-u^N(s)\rangle_{W^*,W}. 
\end{align*}
Inserting these relations in equality (\ref{energy Lemma}) we obtain
\begin{align*}
    &\sigma(t)\lVert P^N u(t)-u^N(t)\rVert_V^2+\frac{2\nu}{C_p^2+\alpha^2} \int_0^t \sigma(s)\lVert P^Nu(s)-u^N(s)\rVert_V^2ds \notag\\ &\leq 2\nu\int_0^t \sigma(s)\langle \hat{A}(u(s)-u^N(s)),u(s)-P^Nu(s)\rangle_Vds\\ & +\int_0^tds \sigma(s)(2C^G_{\alpha,\{\sigma_k\}_{k\in K}}\lVert u^N(s)-P^N u(s)\rVert_V^2+2C^G_{\alpha,\{\sigma_k\}_{k\in K}}\lVert u(s)-P^N u(s)\rVert_V^2)\\ &+2\sum_{k\in K} \int_0^t \sigma(s)\langle \hat{G}^k(u(s))-\hat{G}^{k,N}(s),P^Nu(s)-u^N(s)\rangle_V dW^k_s \notag\\ & -\eta_1\int_0^t \sigma(s)\lVert P^Nu(s)-u^N(s)\rVert_V^2ds-\eta_2\int_0^t \sigma(s)\lVert P^Nu(s)-u^N(s)\rVert_V^2\lVert u(s)\rVert_W^2ds\notag\\ & + \int_0^t \sigma(s)\bigg(C_B^2\lVert  P^Nu(s)-u^N(s)\rVert_V^2\lVert P^N u(s)\rVert_W^2+\lVert  P^Nu(s)-u^N(s)\rVert_V^2\\ & +2\langle \hat{B}(P^Nu(s),P^Nu(s))-\hat{B}^*(s),P^Nu(s)-u^N(s)\rangle_{W^*,W}\bigg) ds \notag\\& +\int_0^t \sigma(s)(4C^F_{\alpha,\{\sigma_k\}_{k\in K}}\lVert u^N(s)-P^N u(s)\rVert_V^2+4C^F_{\alpha,\{\sigma_k\}_{k\in K}}\lVert u(s)-P^N u(s)\rVert_V^2) ds.    
\end{align*}
Taking $\eta_1=2C^G_{\alpha,\{\sigma_k\}_{k\in K}}+4C^F_{\alpha,\{\sigma_k\}_{k\in K}}+1$, $\eta_2=C_B^2$ we get
\begin{align}\label{energy lemma 2}
    &\sigma(t)\lVert P^N u(t)-u^N(t)\rVert_V^2+\frac{2\nu}{C_p^2+\alpha^2} \int_0^t \sigma(s)\lVert P^Nu(s)-u^N(s)\rVert_V^2ds \notag\\ &\leq 2\nu\int_0^t \sigma(s)\langle \hat{A}(u(s)-u^N(s)),u(s)-P^Nu(s)\rangle_Vds +C\int_0^t\ \sigma(s)\lVert u(s)-P^N u(s)\rVert_V^2ds\notag\\ &+2\sum_{k\in K} \int_0^t \sigma(s)\langle \hat{G}^k(u(s))-\hat{G}^{k,N}(s),P^Nu(s)-u^N(s)\rangle_V dW^k_s \notag\\ & +2 \int_0^t \sigma(s) \langle \hat{B}(P^Nu(s),P^Nu(s))-\hat{B}^*(s),P^Nu(s)-u^N(s)\rangle_{W^*,W} ds.
\end{align}
Considering the expected value of \eqref{energy lemma 2} for $t=\tau_M\wedge r$, $r\in [0,T]$, the stochastic integral cancel out, thus we arrive at
\begin{align}\label{crucial lemma behavior N inequality}
    &\mathbb{E}\left[\sigma(\tau_M\wedge r)\lVert P^N u(\tau_M\wedge r)-u^N(\tau_M\wedge r)\rVert_V^2\right]+\frac{2\nu}{C_p^2+\alpha^2} \mathbb{E}\left[\int_0^{\tau_M\wedge r} \sigma(s)\lVert P^Nu(s)-u^N(s)\rVert_V^2ds\right] \notag\\ &\leq 2\nu\mathbb{E}\left[\int_0^{\tau_M\wedge r} \sigma(s)\langle \hat{A}(u(s)-u^N(s)),u(s)-P^Nu(s)\rangle_Vds\right] +C\mathbb{E}\left[\int_0^{\tau_M\wedge r}\ \sigma(s)\lVert u(s)-P^N u(s)\rVert_V^2ds\right]\notag\\ & +2\mathbb{E}\left[ \int_0^{\tau_M\wedge r}\sigma(s) \langle \hat{B}(P^Nu(s),P^Nu(s))-\hat{B}^*(s),P^Nu(s)-u^N(s)\rangle_{W^*,W} ds\right] .
\end{align}
We want understand the behavior of the last term in the inequality above. From Lemma \ref{energy in W} and relation \eqref{preliminary observation P^N} we have 
\begin{align}\label{weak convergence crucial lemma}
    P^Nu-u^N=(P^Nu-u)+(u-u^N)\rightharpoonup 0 \text{ in } L^2(\Omega,\mathcal{F},\mathbb{P};L^2(0,T;W)). \end{align}
Instead we have
\begin{align}\label{strong convergence crucial lemma}
    \left\lVert 1_{[0,\tau_M\wedge r]}\sigma\left(\hat{B}(P^Nu,P^Nu)-\hat{B}(u,u)\right)\right\rVert_{L^2(\Omega,\mathcal{F},\mathbb{P};L^2(0,T;W^*))}\rightarrow 0.
\end{align} 
In fact thanks to relation \eqref{preliminary observation P^N} and the boundedness properties of $\hat{B}$ \eqref{inequality bilinear 1},\eqref{inequality bilinear 2}, $\mathbb{P}-a.s.$ for each $t\in[0,T]$ it holds
\begin{align*}
    \lVert (1_{[0,\tau_M\wedge r]}\sigma)\hat{B}(P^Nu,P^Nu)-\hat{B}(u,u)\rVert_{L^2(0,T;W^*)}\leq C\lVert u\rVert_{L^4(0,T;W)}\lVert (P^N-I)u\rVert_{L^4(0,T;W)}\rightarrow 0.
\end{align*}
Moreover
\begin{align*}
    \lVert (1_{[0,\tau_M\wedge r]}\sigma)\hat{B}(P^Nu,P^Nu)-\hat{B}(u,u)\rVert_{L^2(0,T;W^*)}\leq C\lVert u(t)\rVert_W^2\in L^2(\Omega,\mathcal{F},\mathbb{P};L^2(0,T)).
\end{align*}
By dominated convergence Theorem we have the validity of relation \eqref{strong convergence crucial lemma}. Combing the weak convergence guaranteed by relation \eqref{weak convergence crucial lemma} and the strong convergence guaranteed by \eqref{strong convergence crucial lemma} we obtain
\begin{align*}
    2\mathbb{E}\left[ \int_0^{\tau_M\wedge r}\sigma(s) \langle \hat{B}(P^Nu(s),P^Nu(s))-\hat{B}(u(s),u(s)),P^Nu(s)-u^N(s)\rangle_{W^*,W} ds\right]\rightarrow 0.
\end{align*}
From this relation, by triangle inequality, we can analyze easily the last term in \eqref{crucial lemma behavior N inequality}
\begin{align}\label{convergence nonlinear crucial lemma}
    &\mathbb{E}\left[ \int_0^{\tau_M\wedge r}\sigma(s) \langle \hat{B}(P^Nu(s),P^Nu(s))-\hat{B}^*(s),P^Nu(s)-u^N(s)\rangle_{W^*,W} ds\right]\notag\\ &=\mathbb{E}\left[ \int_0^{\tau_M\wedge r}\sigma(s) \langle \hat{B}(P^Nu(s),P^Nu(s))-\hat{B}(u(s),u(s)),P^Nu(s)-u^N(s)\rangle_{W^*,W} ds\right]\notag\\ &+\mathbb{E}\left[ \int_0^{\tau_M\wedge r}\sigma(s) \langle \hat{B}(u(s),u(s))-\hat{B}^*(s),P^Nu(s)-u^N(s)\rangle_{W^*,W} ds\right]\rightarrow 0.
\end{align}
Thanks to the boundedness of $u^N$ and relation \eqref{preliminary observation P^N}
\begin{align}\label{convergence linear crucial lemma}
    \mathbb{E}\left[\int_0^{\tau_M\wedge r} \sigma(s)\langle \hat{A}(u(s)-u^N(s)),u(s)-P^Nu(s)\rangle_Vds\right]&\leq \lVert u(s)-P^N u(s)\rVert_{L^2(\Omega,\mathcal{F},\mathbb{P};L^2(0,T;V))} \lVert \hat{A}(u-u^N)\rVert_{L^2(\Omega,\mathcal{F},\mathbb{P};L^2(0,T;V))}\notag\\ &\leq C \lVert u-P^N u\rVert_{L^2(\Omega,\mathcal{F},\mathbb{P};L^2(0,T;V))} \rightarrow 0.
\end{align}
Combining \eqref{convergence nonlinear crucial lemma} and \eqref{convergence linear crucial lemma} in relation \eqref{crucial lemma behavior N inequality} we obtain
\begin{align}\label{crucial estimate}
    &\mathbb{E}\left[\sigma(\tau_M\wedge r)\lVert P^N u(\tau_M\wedge r)-u^N(\tau_M\wedge r)\rVert_V^2\right]+\frac{2\nu}{C_p^2+\alpha^2} \mathbb{E}\left[\int_0^{\tau_M\wedge r} \sigma(s)\lVert P^Nu(s)-u^N(s)\rVert_V^2ds\right] \rightarrow 0.
\end{align}
From relation (\ref{crucial estimate}), $\sigma(t)\geq C_M>0\ \forall t\leq \tau_M $ and the properties of $P^N$ via triangle inequality the thesis follows considering $r=T$.
\end{proof}
\begin{remark}\label{changes crucial lemma}
    The proof presented above works only in the case $\nu>0.$ In order to treat the case $\nu=0$ we start from relation \eqref{crucial estimate}. Then, triangle inequality allows to prove 
    \begin{align}\label{crucial estimate nu=0 2}
        \mathbb{E}\left[\sigma(\tau_M\wedge r)\lVert  u(\tau_M\wedge r)-u^N(\tau_M\wedge r)\rVert_V^2\right] \rightarrow 0\quad  \forall r\in [0,T].
    \end{align}
    By dominated convergence theorem we can improve the pointwise convergence of relation \eqref{crucial estimate nu=0 2} in order to obtain Lemma \ref{crucial lemma}. We omit the easy details at this stage, since this argument will be described in full details in the proof of Corollary \ref{corollary convergence} below.
\end{remark}

Combing Lemma \ref{crucial lemma} and the moment estimates for $u$ and $u^N$ we get the following Corollary.
\begin{corollary}\label{corollary convergence}
The subsequence $u^N$ satisfies \begin{align}
    \lim_{N\rightarrow +\infty}\mathbb{E}\left[\lVert u^N(t)-u(t)\rVert_V^2\right]=0, \label{convergence galerkin 1}\\ 
    \lim_{N\rightarrow +\infty}\int_0^T\mathbb{E}\left[\lVert u^N(t)-u(t)\rVert_V^2\right]\ dt=0 \label{convergence galerkin 2}.
\end{align}
\end{corollary}
\begin{proof}
By relation \eqref{crucial estimate} and triangle inequality we already know that 
\begin{align}
    \lim_{N\rightarrow +\infty}\mathbb{E}\left[\lVert u^N(t\wedge \tau_M)-u(t\wedge \tau_M)\rVert_V^2\right]=0 \label{weaker convergence galerkin 1},\\ 
    \lim_{N\rightarrow +\infty}\mathbb{E}\left[\int_0^{T\wedge\tau_M}\lVert u^N(t)-u(t)\rVert_V^2 dt\right]=0 \label{weaker convergence galerkin 2}.
\end{align}
We start proving convergence \eqref{convergence galerkin 1}. By definition of $\tau_M$, Lemma \ref{energy in W} and the weak-$*$ convergence of $u^N$ to $u$ described by relation \eqref{convergences} and Markov's inequality it follows that
\begin{align*}
    \mathbb{E}\left[\lVert u^N(t)-u(t)\rVert_V^2\right]&=\mathbb{E}\left[\lVert u^N(t)-u(t)\rVert_V^2 1_{\tau_M\geq t}\right]+\mathbb{E}\left[\lVert u^N(t)-u(t)\rVert_V^2 1_{\tau_M<t}\right]\\ & = \mathbb{E}\left[\lVert u^N(t\wedge \tau_M)-u(t\wedge \tau_M)\rVert_V^2 1_{\tau_M\geq t}\right]+\mathbb{E}\left[\lVert u^N(t)-u(t)\rVert_V^2 1_{\tau_M<t}\right]\\ & \leq
    \mathbb{E}\left[\lVert u^N(t\wedge \tau_M)-u(t\wedge \tau_M)\rVert_V^2\right]+\mathbb{E}\left[\lVert u^N(t)-u(t)\rVert_V^4 \right]^{1/2}\mathbb{P}(\tau_M<t)^{1/2}\\ &  \leq
    \mathbb{E}\left[\lVert u^N(t\wedge \tau_M)-u(t\wedge \tau_M)\rVert_V^2\right]+C\operatorname{sup}_{N\in\mathbb{N}}\mathbb{E}\left[\lVert u^N(t)\rVert_W^4\right]^{1/2}\mathbb{P}(\operatorname{sup}_{t\in[0,T]}\lVert u(t)\rVert_W>M)^{1/2}\\ & \leq
    \mathbb{E}\left[\lVert u^N(t\wedge \tau_M)-u(t\wedge \tau_M)\rVert_V^2\right]+\frac{C}{M^2}\operatorname{sup}_{N\in\mathbb{N}}\mathbb{E}\left[\operatorname{sup}_{t\in [0,T]}\lVert u^N(t)\rVert_W^4\right]\\ & \leq \mathbb{E}\left[\lVert u^N(t\wedge \tau_M)-u(t\wedge \tau_M)\rVert_V^2\right]+\frac{C_{\nu, \alpha,\{\sigma_k\}_{k\in K}}}{M^2} ,
\end{align*}
where $C_{\nu, \alpha,\{\sigma_k\}_{k\in K}}$ is a constant independent from $M$ and $N$. If we fix $\epsilon>0$ and choose $M$ large enough such that $\frac{C_{\nu, \alpha,\{\sigma_k\}_{k\in K}}}{M^2}\leq \epsilon$ then by relation \eqref{weaker convergence galerkin 1} we have \begin{align*}
    \operatorname{limsup}_{N\rightarrow +\infty}\mathbb{E}\left[\lVert u^N(t)-u(t)\rVert_V^2\right]\leq \epsilon.
\end{align*}
From the arbitrariness of $\epsilon$, the first thesis follows. In order to obtain the other convergence we apply dominated convergence Theorem. Indeed, by relation \eqref{convergence galerkin 1} we already know that for each $t\in [0,T]$ \begin{align*}
    \mathbb{E}\left[\lVert u^N(t)-u(t)\rVert_V^2\right]\rightarrow 0. 
\end{align*} Moreover, by Lemma \eqref{energy in V}, for each $N$ \begin{align*}
    \mathbb{E}\left[\lVert u^N(t)-u(t)\rVert_V^2\right]&\leq 2 \mathbb{E}\left[\lVert u^N(t)\rVert_V^2\right]+2\mathbb{E}\left[\lVert u(t)\rVert_V^2\right]\\ & \leq C_{p,\alpha,\{\sigma_k\}_{k\in K}}+2\mathbb{E}\left[\lVert u(t)\rVert_V^2\right]\in L^1(0,T).
\end{align*}
Therefore convergence \eqref{convergence galerkin 2} follows.
\end{proof}

From Lemma \ref{crucial lemma}, without any change with respect to the proof of Lemma 3.8 in \cite{razafimandimby2012strong}, we have that the Lemma below holds, thus $u$ is a solution of problem (\ref{Ito}) in the sense of Definition \ref{weak solution}.
\begin{Lemma}
$\hat{B}^*=\hat{B}(u,u) \textit{\ in\ } L^2(\Omega,\mathcal{F},\mathbb{P};L^2(0,T;W^*))$ 
\end{Lemma}
Now we can prove the uniqueness.
\begin{theorem}
The solution of problem (\ref{Ito}) in the sense of Definition \ref{weak solution} is unique.\end{theorem}
\begin{proof}
Let $u_1$ and $u_2$ be two solutions. Let $w$ be their difference, then for each $\phi \in W$ and $t>0$
\begin{align*}
    \langle w(t), \phi\rangle_V&=\nu \int_0^t \langle \nabla w(s), \nabla\phi\rangle_{L^2} ds- \int_0^t b(u_1(s), u_1(s)-\alpha^2 \Delta u_1(s),\phi)ds\\ &-\alpha^2 \int_0^t b(\phi, \Delta u_1(s),u_1(s))ds+\int_0^t b(u_2(s), u_2(s)-\alpha^2 \Delta u_2(s),\phi)ds\\ &+\alpha^2 \int_0^t b(\phi, \Delta u_2(s),u_2(s))ds+\int_0^t \langle F(w(s)),\phi \rangle ds \\ &+ \sum_{k\in K} \int_0^t \langle G^k(w(s)),\phi \rangle  dW^k_s \ \ \mathbb{P}-a.s.
\end{align*}
Now we apply the It\^o's formula to compute $\lVert w\rVert_V^2$. Arguing as in the first part of the proof of Lemma \ref{thm energy inequality} we obtain \begin{align*}
    d\lVert w\rVert_V^2=-2\nu\lVert \nabla w\rVert_{L^2}^2 dt +(b(w,w-\alpha\Delta w, u_2)-b(u_2,w-\alpha\Delta w, w))dt.
\end{align*}
Let us consider $exp(-\int_0^t \lVert u_2(s)\rVert_W^2ds)\lVert w(t)\rVert_V^2:=\sigma(t)\lVert w(t)\rVert_V^2$, via It\^o's formula we get\begin{align*}
    d(\sigma\lVert w\rVert_V^2)+2\nu\sigma\lVert \nabla w\rVert_{L^2}^2 dt=-\sigma\lVert u_2\rVert_W^2\lVert w\rVert_V^2 dt+\sigma(b(w,w-\alpha\Delta w, u_2)-b(u_2,w-\alpha\Delta w, w))dt
\end{align*}
Combining relations \eqref{equvalence hatB and b} and \eqref{inequality trilinear 2} it follows that  \begin{align*}
    \lvert b(w,w-\alpha^2\Delta w,u_2)-b(u_2,w-\alpha^2\Delta w,w)\rvert\leq C\lVert w\rVert_V^2\lVert u_2 \rVert_W. 
\end{align*}
Therefore
\begin{align*}
    d(\sigma\lVert w\rVert_V^2)\leq -\sigma\lVert u_2\rVert_W^2\lVert w\rVert_V^2 dt+C\sigma\lVert w \rVert_V^2\lVert u_2\rVert_Wdt\leq C_{\epsilon}\sigma\lVert w\rVert_V^2,
\end{align*}
where in the last step we applied Young's inequality. From the last chain of inequalities, via Gr\"onwall's Lemma we get the thesis.
\end{proof}
\begin{theorem}\label{theorem convergence galerkin}
The entire Galerkin's sequence $u^N$ satisfies \begin{align*}
    \lim_{N\rightarrow +\infty}\mathbb{E}\left[\lVert u^N(t)-u(t)\rVert_V^2\right]=0,\\ 
    \lim_{N\rightarrow +\infty}\int_0^T\mathbb{E}\left[\lVert u^N(t)-u(t)\rVert_V^2\right]\ dt=0.
\end{align*}
\end{theorem}
\begin{proof}
Each subsequence $u^{N_k}$ has a converging sub-subsequence $u^{N_{k,k}}$ which satisfies all previous Lemmas. By uniqueness of the solution of equation \eqref{Ito} and Corollary \ref{corollary convergence} then the thesis follows.
\end{proof}
\begin{remark}
Theorem \ref{theorem convergence galerkin} plays no role concerning the well-posedness of equation (\ref{Ito}), but it will be crucial for obtaining the energy estimates of Section \ref{energy estimates}, and thus for proving Theorem \ref{Main thm}.
\end{remark}
\section{Energy Estimates}\label{energy estimates}
Now we start considering equations (\ref{Ito Scaled}) and assuming also Hypothesis \ref{Hypothesis inviscid limit}. The goal of this Section is to prove the following lemma:

\begin{Lemma}\label{thm energy inequality}
Under Hypothesis \ref{hypothesis well-posedness}-\ref{Hypothesis inviscid limit}, if $u^{\alpha}$ is the solution of problem (\ref{Ito Scaled}) in the sense of Definition \ref{weak solution}, then \begin{align}
    &\lVert u^{\alpha}(t)\rVert^2+\alpha^2\lVert \nabla u^{\alpha}(t)\rVert_{L^2}^2+\nu\int_0^t \lVert \nabla u^{\alpha}(s)\rVert_{L^2}^2ds=\lVert u^{\alpha}_0\rVert^2+\alpha^2\lVert \nabla u^{\alpha}_0\rVert_{L^2}^2\label{Ito solution};\\
    &\mathbb{E}\left[\alpha^6\operatorname{sup}_{t\in [0,T]}\lVert u^{\alpha}(t)\rVert_{H^3}^2\right]=O(1). \label{H3 scaling solution}
\end{align}
\end{Lemma}
\begin{proof}
For the sake of simplicity we write $u$ and $u_0$ instead of $u^{\alpha},\ u^{\alpha}_0$ since $\alpha$ is fixed in this proof. Therefore all the asymptotic expansions and limits will be considering $N\rightarrow +\infty.$
\begin{itemize}
    \item Let $\Tilde{e}_i$ be the eigenfunctions of the Stokes operator $-A$, and $\Tilde{\lambda}_i$ the corresponding eigenvalues introduced in Lemma \ref{eigenfunction Stokes}. Let, moreover, $\Tilde{u}^N=\sum_{i=1}^N\langle u,\Tilde{e}_i\rangle \Tilde{e}_i=\Tilde{P}^N u$. Exploiting the weak formulation with test functions $\Tilde{e}_i$ we get
\begin{align*}
    \langle u(t), \Tilde{e}_i\rangle-\alpha^2\langle u(t), A\Tilde{e}_i\rangle-  \langle u_0, \Tilde{e}_i\rangle+\alpha^2 \langle u_0, A\Tilde{e}_i\rangle&=\nu \int_0^t \langle u(s), A\Tilde{e}_i\rangle ds- \int_0^t b(u(s), u(s)-\alpha^2 \Delta u(s),\Tilde{e}_i)ds\\ &-\alpha^2 \int_0^t b(\Tilde{e}_i, \Delta u(s),u(s))ds+\Tilde{\nu}\int_0^t \langle F(u),\Tilde{e}_i \rangle ds \\ &+ \sqrt{\Tilde{\nu}}\sum_{k\in K} \int_0^t \langle G^k(u(s)),\Tilde{e}_i \rangle  dW^k_s \ \ \mathbb{P}-a.s.
\end{align*}
Multiplying each equation by $\Tilde{e}_i$ and summing up, we get
\begin{align*}
    d(\Tilde{u}^N-\alpha^2A\Tilde{u}^N)&=\nu A\Tilde{u}^N dt-\sum_{i=1}^N b(u, u-\alpha^2 \Delta u,\Tilde{e}_i)\ dt\\ &-\alpha^2 \sum_{i=1}^N\int_0^t b(\Tilde{e}_i, \Delta u,u) \Tilde{e}_i\ dt+ \Tilde{\nu}\sum_{i=1}^N\langle F(u),\Tilde{e}_i \rangle \Tilde{e}_i \ dt \\ &+ \sqrt{\Tilde{\nu}}\sum_{k\in K} \sum_{i=1}^N \langle G^k(u),\Tilde{e}_i \rangle \Tilde{e}_i \ dW^k_t
\end{align*}
Now we can apply the It\^o's formula to the process \begin{align*}
    \frac{1}{2} (\lVert \Tilde{u}^N(t)\rVert^2+\alpha^2 \lVert \nabla \Tilde{u}^N(t)\rVert_{L^2}^2)=\frac{1}{2} \langle (I-\alpha^2 A)\Tilde{u}^N(t),\Tilde{u}^N(t)\rangle
\end{align*}
obtaining \begin{align*}
   \frac{\lVert \Tilde{u}^N(t)\rVert^2+\alpha^2 \lVert  \nabla \Tilde{u}^N(t)\rVert}{2}&=\frac{\lVert \Tilde{u}^N_0\rVert^2+\alpha^2 \lVert  \nabla \Tilde{u}^N_0\rVert}{2}  -\nu\int_0^t \langle \nabla \Tilde{u}^N(s), \nabla u^N(s)\rangle_{L^2} ds-\int_0^t b(u(s),u(s)\\ &-\alpha^2\Delta u(s),\Tilde{u}^N(s))-\alpha^2\int_0^t b(\Tilde{u}^N(s),\Delta u(s), u(s))ds\\ & +\frac{\Tilde{\nu}}{2}\int_0^t \sum_{k\in K} \langle P(\sigma_k\cdot \nabla((I-\alpha^2 A)^{-1}P(\sigma_k\cdot\nabla u(s)))),\Tilde{u}^N(s)\rangle ds\\ &+ \sqrt{\Tilde{\nu}}\sum_{k\in K}\int_0^t \langle P(\sigma_k\cdot \nabla u(s)),\Tilde{u}^N\rangle dW^K_s\\ &+\frac{\Tilde{\nu}}{2}\sum_{k\in K}\int_0^t \sum_{i=1}^N\langle P(\sigma_k\cdot\nabla u(s)),\tilde{e}_i\rangle^2\langle \tilde{e}_i, (I-\alpha^2A)^{-1}\tilde{e}_i\rangle ds. 
\end{align*}
Thanks to the properties of the projector $\Tilde{P}^N$ we get easily the first relation. The only thing we need to prove is that \begin{align*}
    \sum_{i=1}^N\langle P(\sigma_k\cdot\nabla u),\tilde{e}_i\rangle^2\langle \tilde{e}_i, (I-\alpha^2A)^{-1}\tilde{e}_i\rangle+\langle P(\sigma_k\cdot \nabla((I-\alpha^2 A)^{-1}P(\sigma_k\cdot\nabla u))),\Tilde{u}^N\rangle\rightarrow 0.
\end{align*} The last relation is true, in fact 
\begin{align*}
    &\sum_{i=1}^N\langle P(\sigma_k\cdot\nabla u),\tilde{e}_i\rangle^2\langle \tilde{e}_i, (I-\alpha^2A)^{-1}\tilde{e}_i\rangle+\langle P(\sigma_k\cdot \nabla((I-\alpha^2 A)^{-1}P(\sigma_k\cdot\nabla u))),\Tilde{u}^N\rangle=\\& \sum_{i=1}^N\langle P(\sigma_k\cdot\nabla u),(I-\alpha^2 A)^{-1/2}\tilde{e}_i\rangle^2+\langle P(\sigma_k\cdot \nabla((I-\alpha^2 A)^{-1}P(\sigma_k\cdot\nabla u))),\Tilde{u}^N\rangle\\ & \rightarrow \langle (I-\alpha^2 A)^{-1}P(\sigma_k\cdot \nabla u),P(\sigma_k\cdot \nabla u)\rangle+\langle P(\sigma_k\cdot \nabla((I-\alpha^2 A)^{-1}P(\sigma_k\cdot\nabla u))),u\rangle=0.
\end{align*}
\item From Theorem \ref{theorem convergence galerkin} and equation \eqref{equivalence H1-V}, we know that \begin{align*}
    \int_0^T \mathbb{E}\left[\lVert \nabla u^N(s)\rVert_{L^2}^2\right] ds& \leq \frac{1}{\alpha^2}\int_0^T \mathbb{E}\left[\lVert u^N(s)\rVert_V^2\right]ds =\frac{1}{\alpha^2}\int_0^T\mathbb{E}\left[\lVert u(s)\rVert_V^2\right] ds+o(1).
\end{align*}
Thus, from the It\^o formula \eqref{Ito solution} the following relations hold true:
\begin{align}\label{importanti ito}
   & \int_0^T \mathbb{E}\left[\lVert \nabla u^N(s)\rVert_{L^2}^2 \right] ds\leq \frac{C}{\alpha^2}\mathbb{E}\left[\lVert  u_0\rVert^2\right]+C\mathbb{E}\left[\lVert \nabla u_0\rVert_{L^2}^2\right]+o(1)\\\label{importanti ito2}
    & \mathbb{E}\left[\operatorname{sup}_{t\in [0,T]}\lVert \nabla u(t)\rVert_{L^2}^2 \right]\leq \frac{1}{\alpha^2}\mathbb{E}\left[\lVert  u_0\rVert^2\right]+\mathbb{E}\left[\lVert \nabla u_0\rVert_{L^2}^2\right]
\end{align}
According to inequality (\ref{scaling H^3}), in order to prove relation \eqref{H3 scaling solution}, it remains to study \begin{align*}\
    \mathbb{E}\left[\operatorname{sup}_{t\in [0,T]}\lVert  u^N(t)\rVert_*^2 \right].
\end{align*}
Before going on we recall some notation. For each $N\in\mathbb{N}$
\begin{align*}
    \tau_M^N=\inf\{t:\ \lVert u^N(t)\rVert_V+\lVert u^N(t)\rVert_*\geq M \}\wedge T,
\end{align*}
Thanks to the scaling factor $\sqrt{\Tilde{\nu}}$ appearing in front of the noise and exploiting the asymptotic relation between $\nu,\ \Tilde{\nu}$ and $\alpha^2$ described by Hypothesis \ref{Hypothesis inviscid limit} , if we choose \begin{align*}
    \epsilon=\frac{1}{2\sum_{k\in K}\lVert \sigma_k\rVert_{W^{1,\infty}}^2},
\end{align*} equation \eqref{final norm W 2} in Lemma \ref{energy in W} becomes
\begin{align}\label{scaling norm W 2}
    &\mathbb{E}\left[\operatorname{sup}_{t\leq r\wedge \tau_M^N}\lVert u^N(t)\rVert_*^2\right]+\mathbb{E}\left[\int_0^{r\wedge \tau_M^N} \lVert u^N(s)\rVert_*^2ds\right] \leq C\left(\mathbb{E}\left[\lVert u^N_0\rVert_*^2\right]+\left(\alpha^2+1\right) \mathbb{E}\left[\int_0^{r\wedge \tau_M^N} \lVert \nabla u^{N}(s)\rVert_{L^2}^2\right]\right).
\end{align}
Therefore, thanks to equation \eqref{importanti ito}, we have 
\begin{align}\label{scaled norm W 2}
     &\mathbb{E}\left[\operatorname{sup}_{t\leq r\wedge \tau_M^N}\lVert u^N(t)\rVert_*^2\right]+\mathbb{E}\left[\int_0^{r\wedge \tau_M^N} \lVert u^N(s)\rVert_*^2ds\right]\notag\\ & \leq C\left(\mathbb{E}\left[\lVert u^N_0\rVert_*^2\right]+\left(\alpha^2+1\right)\left( \mathbb{E}\left[\frac{\lVert  u_0\rVert^2}{\alpha^2}\right]+\mathbb{E}\left[\lVert \nabla u_0\rVert_{L^2}^2\right]\right)\right)+o(1).
\end{align}

So far we showed that $u^N\in L^2(\Omega;L^2([0,T];H^1))$, $\operatorname{curl}(u^N-\alpha^2\Delta u^N)\in L^2(\Omega, L^{\infty}([0,T];L^2))$. By monotone convergence Theorem, we can remove the dependence from $M$ in relation \eqref{scaled norm W 2}. Therefore
\begin{align}
     &\mathbb{E}\left[\operatorname{sup}_{t\leq T}\lVert u^N(t)\rVert_*^2\right] \leq C\left(\mathbb{E}\left[\lVert u^N_0\rVert_*^2\right]+\left(\alpha^2+1\right)\left( \mathbb{E}\left[\frac{\lVert  u_0\rVert^2}{\alpha^2}\right]+\mathbb{E}\left[\lVert \nabla u_0\rVert_{L^2}^2\right]\right)\right)+o(1)\label{final scaling estimate 2}.
\end{align}
Thus, by Theorem \ref{theorem convergence galerkin} and the uniform bound \eqref{final scaling estimate 2} there exists a subsequence $N_k$ such that \begin{align*}
    & u^{N_k}\rightarrow u\text{ in } L^2(\Omega;L^2([0,T];H^1))\\
    & \operatorname{curl}(u^{N_k}-\alpha^2\Delta u^{N_k})\stackrel{*}{\rightharpoonup} g \text{ in } L^2(\Omega, L^{\infty}([0,T];L^2)).
\end{align*} If we take a test function $\phi\in L^2(\Omega;L^2(0,T;C^{\infty}_c(D)))$, we get easily \begin{align*}
    \mathbb{E}\left[\int_0^T \langle \phi(s),g(s)\rangle_{L^2} ds \right]&=\lim_{k\rightarrow+\infty}\mathbb{E}\left[\int_0^T \langle \phi(s),\operatorname{curl}(u^{N_k}(s)-\alpha^2\Delta u^{N_k}(s))\rangle_{L^2} ds\right]\\ & = \lim_{k\rightarrow+\infty}\mathbb{E}\left[\int_0^T \langle (I-\alpha^2\Delta)\nabla^{\perp}\phi(s),u^{N_k})(s)\rangle_{L^2} ds \right]\\ & =\mathbb{E}\left[\int_0^T \langle (I-\alpha^2\Delta)\nabla^{\perp}\phi(s),u(s))\rangle_{L^2} ds\right].
\end{align*}
Therefore $g=\operatorname{curl}(u-\alpha^2\Delta u)\in L^2(\Omega, L^{\infty}([0,T];L^2)) $ and the following inequality holds true \begin{align}\label{prefinal norm *}
    \mathbb{E}\left[\operatorname{sup}_{t\leq T}\lVert u(t)\rVert_{*}^2\right]\leq
    C\left(\liminf_{k\rightarrow +\infty}\mathbb{E}\left[\lVert u^{N_k}_0\rVert_*^2\right]+\left(\alpha^2+1\right)\left( \mathbb{E}\left[\frac{\lVert  u_0\rVert^2}{\alpha^2}\right]+\mathbb{E}\left[\lVert \nabla u_0\rVert_{L^2}^2\right]\right)\right).
\end{align}
Let us analyze better the first term. We denote by $u_0^{N,\infty}=u_0-u_0^N$. 
\begin{align*}
    \lVert u^{N_k}_0\rVert_*^2&=\lVert u^{N_k}_0\rVert_W^2-\lVert u^{N_k}_0\rVert_V^2\\ &\leq \lVert u_0\rVert_W^2-\lVert u^{N_k}_0\rVert_V^2\\ & \leq \lVert u_0\rVert_*^2+\lVert u^{{N_k},\infty}_0\rVert_V^2\\ & \leq \lVert u_0\rVert_*^2+\lVert u_0\rVert_V^2\\ &\leq C( \lVert \nabla u_0\rVert_{L^2}^2+\alpha^4\lVert \operatorname{curl}\Delta u_0\rVert_{L^2}^2+\lVert u_0\rVert^2+\alpha^2\lVert \nabla u_0\rVert_{L^2}^2)\\ & \leq C(
    \lVert u_0\rVert^2+(1+\alpha^2)\lVert \nabla u_0\rVert_{L^2}^2+\alpha^4\lVert u_0\rVert_{H^3}^2)
\end{align*}
In conclusion, combining the observation above,  relations \eqref{scaling H^3}, (\ref{importanti ito2}) and \eqref{prefinal norm *} we get

\begin{align}\label{explicit estimate H^3}
    \mathbb{E}\left[\operatorname{sup}_{t\leq T}\lVert u(t)\rVert_{H^3}^2\right]& \leq C\left( \frac{\alpha^4+\alpha^2+1}{\alpha^6}\mathbb{E}\left[\lVert  u_0\rVert^2\right]+\frac{1+\alpha^2+\alpha^4}{\alpha^4}\mathbb{E}\left[\lVert \nabla u_0\rVert_{L^2}^2\right]+ \mathbb{E}\left[\lVert u_0\rVert_{H^3}^2\right]\right) .\end{align}
\end{itemize}
Thanks to the assumptions on $u_0^{\alpha}$, see Hypothesis \ref{Hypothesis inviscid limit}, the thesis follows.
\end{proof}
\begin{remark}\label{remark thm energy inequality}
In the case $\nu=0$, relation \eqref{Ito solution} follows without any change with respect to the main proof. For what concerns relation \eqref{H3 scaling solution}, equation \eqref{scaling norm W 2} above is false in this framework. However, introducing the proper scaling in front of the noise we can restart from relation \eqref{ito formula norm * galerkin step 3 remark} obtaining \begin{align}
    \mathbb{E}\left[\operatorname{sup}_{t\leq r\wedge \tau_M^N}\lVert u^N(t)\rVert_*^2\right]& \leq 2\mathbb{E}\left[\lVert u^N_0\rVert_*^2\right]+\mathbb{E}\left[\int_0^{r\wedge \tau_M^N} \lVert u^N(s)\rVert_*^2 ds\right]\notag\\ & +4\sqrt{\Tilde{\nu}}\mathbb{E}\left[\operatorname{sup}_{t\leq r\wedge \tau_M^N}\left\lvert \sum_{k\in K} \int_0^{t}\langle \operatorname{curl}({G}^{k,N}(s) ),\operatorname{curl}(u^N(s)-\alpha^2\Delta u^N(s) )\rangle_{L^2} dW^k_s\right\rvert\right]\notag\\ &+ 4\Tilde{\nu}\mathbb{E}\left[\int_0^{r\wedge \tau_M^N}\lVert \operatorname{curl} F^N(s)\rVert_{L^2}^2 ds \right]+2\Tilde{\nu}\sum_{k\in K}\sum_{i=1}^N(\lambda_i+\lambda_i^2)\mathbb{E}\left[\int_0^{r\wedge \tau_M^N}\langle {G}^{k,N}(s),e_i \rangle^2ds\right].
\label{remark scaling norm W 2}\end{align}
Therefore, combining estimates \eqref{estimate 1},\eqref{estimate 2},\eqref{estimate 3},\eqref{estimate 4}, exploiting the asymptotic relation between $\Tilde{\nu}$ and $\alpha^2$ described by Hypothesis \ref{Hypothesis inviscid limit} and choosing $\epsilon=\frac{1}{\sum_{k\in K}\lVert \sigma_k\rVert_{W^{1,\infty}}^2}$, we obtain
\begin{align}
    \mathbb{E}\left[\operatorname{sup}_{t\leq r\wedge \tau_M^N}\lVert u^N(t)\rVert_*^2\right]& \leq 4\mathbb{E}\left[\lVert u^N_0\rVert_*^2\right]+6\mathbb{E}\left[\int_0^{r\wedge \tau_M^N} \lVert u^N(s)\rVert_*^2 ds\right] +C_{\{\sigma_k\}_{k\in K}}\left(\alpha^2+1\right) \mathbb{E}\left[\int_0^{r\wedge \tau_M^N} \lVert \nabla u^{N}(s)\rVert_{L^2}^2 ds\right].
\label{scaling norm W 2 remark}\end{align}
Therefore, thanks to equation \eqref{importanti ito}, we have 
\begin{align}
    \mathbb{E}\left[\operatorname{sup}_{t\leq r\wedge \tau_M^N}\lVert u^N(t)\rVert_*^2\right]& \leq 4\mathbb{E}\left[\lVert u^N_0\rVert_*^2\right]+6\mathbb{E}\left[\int_0^{r\wedge \tau_M^N} \lVert u^N(s)\rVert_*^2 ds\right]\notag\\ &  +C_{\{\sigma_k\}_{k\in K}}\left(\alpha^2+1\right) \left( \mathbb{E}\left[\frac{\lVert  u_0\rVert^2}{\alpha^2}\right]+\mathbb{E}\left[\lVert \nabla u_0\rVert_{L^2}^2\right]\right)+o(1).
\label{scaled norm W 2 remark}\end{align}
Arguing as in Remark \ref{changes Lemma energy W}, we can apply Gr\"onwall's Lemma in inequality \eqref{scaled norm W 2 remark} obtaining 
\begin{align}\label{scaled norm W 2 final remark}
     &\mathbb{E}\left[\operatorname{sup}_{t\leq r\wedge \tau_M^N}\lVert u^N(t)\rVert_*^2\right]  \leq C_{\{\sigma_k\}_{k\in K}}\left(\mathbb{E}\left[\lVert u^N_0\rVert_*^2\right]+\left(\alpha^2+1\right)\left( \mathbb{E}\left[\frac{\lVert  u_0\rVert^2}{\alpha^2}\right]+\mathbb{E}\left[\lVert \nabla u_0\rVert_{L^2}^2\right]\right)\right)+o(1).
\end{align}
Relation \eqref{scaled norm W 2 final remark} is completely analogous to relation \eqref{scaled norm W 2} above. Therefore we can follow the same argument of the main proof in order to obtain estimate \eqref{H3 scaling solution} and we omit the details.
\end{remark}

\section{Proof of Theorem \ref{Main thm}}\label{proof of main thm}
In order to prove Theorem \ref{Main thm}, we will follow the ideas of \cite{lopes2015approximation} and \cite{luongo2021inviscid}. We will start with a weaker result with the supremum in time outside the expected value and then we will move to the stronger one with the supremum in time inside the expected value. 

\begin{proof}[Proof of Theorem \ref{Main thm}]
Let $W^{\alpha}=u^{\alpha}-\Bar{u}$, it satisfies $\mathbb{P}-a.s.$ for each $\phi \in H$ and $t\in[0,T]$

\begin{align*}
    \langle W^{\alpha}(t),\phi\rangle-\langle W^{\alpha}_0,\phi\rangle& =\alpha^2\langle Au^{\alpha}(t),\phi\rangle-\alpha^2\langle Au^{\alpha}_0,\phi\rangle+\nu \int_0^t \langle A u^{\alpha}(s), \phi\rangle ds\\ & -\int_0^t b(u^{\alpha}(s),W^{\alpha}(s),\phi)ds -\int_0^tb(W^{\alpha}(s),\Bar{u}(s),\phi) ds\\ & -\alpha^2 \int_0^t b(\phi, \Delta u^{\alpha}(s),u^{\alpha}(s))ds+\alpha^2 \int_0^t b(u^{\alpha}(s), \Delta u^{\alpha}(s),\phi)ds\\ & \Tilde{\nu} \int_0^t\langle F(u^{\alpha})(s),\phi\rangle ds +\sqrt{\Tilde{\nu}}\sum_{k\in K}\int_0^t\langle G^k(u^{\alpha}(s)),\phi\rangle d W^k_s.
\end{align*}
Following the idea of \cite{kato1984remarks}, let $v$ the corrector of the boundary layer of width $\delta$, i.e. a divergence free vector field with support in a strip of the boundary of width $\delta$ such that $\Bar{u}-v\in V$ and \begin{align}\label{property bl corrector}
    \operatorname{sup}_{t\in[0,T]}\lVert \partial_t^l v\rVert\lesssim \delta^{1/2},\ \ \ \operatorname{sup}_{t\in[0,T]}\lVert \partial_t^l \nabla v\rVert\lesssim \delta^{-1/2}, \ l\in\{0,1\}.
\end{align}
Let $\delta=\delta(\alpha)$ such that
\begin{align}\label{assumptions delta}
    \lim_{\alpha\rightarrow 0}\delta=0,\ \ \lim_{\alpha\rightarrow 0} \frac{\alpha^2}{\delta}=0.
\end{align}
We want to write the It\^o's formula for $\lVert W^{\alpha}(t)\rVert^2$. Let us take an orthonormal basis of $H$, $\{\Tilde{e}_i\}$ made by eigenvectors of $A$, let $\{-\Tilde{\lambda}_i\}$ the corresponding eigenvalues. Let us consider the weak formulation with test functions $\phi=\Tilde{e}_i$, let us call $W^{\alpha,n}=\sum_{i=1}^n \langle W^{\alpha},\tilde{e}_i\rangle \tilde{e}_i$, $u^{\alpha,n}=\sum_{i=1}^n \langle u^{\alpha},\tilde{e}_i\rangle \tilde{e}_i$, $\Bar{u}^n=\sum_{i=1}^n \langle \Bar{u},\tilde{e}_i\rangle \tilde{e}_i$ e $v^n=\sum_{i=1}^n\langle v,\tilde{e}_i\rangle \tilde{e}_i$, then, arguing as in the proof of Lemma \ref{thm energy inequality}, we get  
\begin{align}\label{ito formula 1}
    W^{\alpha,n}(t)-W^{\alpha,n}_0& =\alpha^2 Au^{\alpha,n}(t)-\alpha^2 Au^{\alpha,n}_0+\nu\int_0^t  Au^{\alpha,n}(s)ds\notag\\ &-\int_0^t \sum_{i=1}^n b(u^{\alpha},W^{\alpha}(s),\tilde{e}_i)\tilde{e}_ids-\int_0^t \sum_{i=1}^n b(W^{\alpha}(s),\Bar{u}(s),\tilde{e}_i)\tilde{e}_ids\notag\\ & -\alpha^2\int_0^t \sum_{i=1}^n b(\tilde{e}_i,\Delta u^{\alpha}(s),u^{\alpha}(s)) ds +\alpha^2\int_0^t \sum_{i=1}^n b(u^{\alpha}(s),\Delta u^{\alpha}(s),\tilde{e}_i) ds\notag\\ & +\Tilde{\nu}\int_0^t\sum_{i=1}^n\langle F(u^{\alpha}(s)),\tilde{e}_i\rangle \tilde{e}_ids+\sqrt{\Tilde{\nu}}\sum_{k\in K}\int_0^t \sum_{i=1}^n\langle G^k(u^{\alpha}(s)), \tilde{e}_i\rangle \tilde{e}_i dW^k_s.
\end{align}
Therefore \begin{align}\label{ito formula 2}
    d\lVert W^{\alpha,n}\rVert^2&=2\langle W^{\alpha,n},dW^{\alpha,n}\rangle+\alpha^4 d\langle\langle Au^{\alpha,n},Au^{\alpha,n}\rangle\rangle_t+\Tilde{\nu} \sum_{k\in K}\sum_{i=1}^n \langle G^k(u^{\alpha}), \tilde{e}_i\rangle^2 dt\notag\\ &+\alpha^2\sqrt{\Tilde{\nu}}\sum_{k\in K}  \sum_{i=1}^n \langle G^k(u^{\alpha}),\tilde{e}_i\rangle d\langle\langle \tilde{e}_i W^k, Au^{\alpha,n}\rangle\rangle_t.\end{align}
In the same way, considering the weak formulation satisfied by $u^{\alpha}$, we get 
\begin{align}\label{ito formula 3}
    d Au^{\alpha,n}&=\left(\nu A(I-\alpha^2 A)^{-1}Au^{\alpha,n}-\sum_{i=1}^n b(u^{\alpha,n},u^{\alpha,n},\tilde{e}_i)A(I-\alpha^2 A)^{-1}\tilde{e}_i\right)dt\notag\\ &-\alpha^2 \sum_{i=1}^n b(\tilde{e}_i,\Delta u,u )A(I-\alpha^2 A)^{-1}\tilde{e}_i dt \notag\\ &+\Tilde{\nu}\sum_{i=1}^n\langle F(u^{\alpha}), \tilde{e}_i\rangle A(I-\alpha^2 A)^{-1}\tilde{e}_i dt\notag\\ &+\sqrt{\Tilde{\nu}}\sum_{k\in K}\sum_{i=1}^n\langle G^k(u^{\alpha}), \tilde{e}_i\rangle A(I-\alpha^2 A)^{-1}\tilde{e}_i dW^k_t.
\end{align}
Combining relation \eqref{ito formula 1}, \eqref{ito formula 2}, \eqref{ito formula 3} we obtain
\begin{align*}
    d\lVert W^{\alpha,n}\rVert^2&=2\alpha^2\langle W^{\alpha,n},dAu^{\alpha,n}\rangle+2\nu\langle W^{\alpha,n},Au^{\alpha,n}\rangle dt\\ &-2\langle W^{\alpha,n},\sum_{i=1}^n b(W^{\alpha},\Bar{u},\tilde{e}_i)\tilde{e}_i\rangle dt-2\langle W^{\alpha,n},\sum_{i=1}^n b(u^{\alpha},W^{\alpha},\tilde{e}_i)\tilde{e}_i\rangle dt\\ & -2\alpha^2\langle W^{\alpha,n},\sum_{i=1}^n b(\tilde{e}_i,\Delta u^{\alpha},u^{\alpha})\tilde{e}_i\rangle dt+2\alpha^2\langle W^{\alpha,n},\sum_{i=1}^n b(u^{\alpha},\Delta u^{\alpha},\tilde{e}_i)\tilde{e}_i\rangle \ dt\\ &+2\Tilde{\nu}\sum_{i=1}^n\langle F(u^{\alpha}),\tilde{e}_i\rangle \langle W^{\alpha,n},\tilde{e}_i\rangle dt+2\sqrt{\Tilde{\nu}}\sum_{k\in K}\sum_{i=1}^n \langle G^{k}(u),\tilde{e}_i\rangle\langle W^{\alpha,n},\tilde{e}_i\rangle dW^k_t\\ &+\alpha^4 \Tilde{\nu} \sum_{k\in K}\sum_{i=1}^n \langle G^k(u^{\alpha}), \tilde{e}_i\rangle^2 \lVert A(I-\alpha^2 A)^{-1}\tilde{e}_i\rVert^2 \ dt\\ &+\Tilde{\nu} \sum_{k\in K}\sum_{i=1}^n\langle G^k(u^{\alpha}), \tilde{e}_i\rangle^2 dt +\alpha^2\Tilde{\nu}\sum_{k\in K}\sum_{i=1}^n \langle G^k(u^{\alpha}), \tilde{e}_i\rangle^2\langle A(I-\alpha^2 A)^{-1}\tilde{e}_i,\tilde{e}_i\rangle dt.
\end{align*}
Let us rewrite $\langle W^{\alpha,n},dAu^{\alpha,n}\rangle$ in a different way
\begin{align*}
    \langle W^{\alpha,n},dAu^{\alpha,n}\rangle&=\langle u^{\alpha,n}-\Bar{u}^n,dAu^{\alpha,n}\rangle=\langle u^{\alpha,n},dAu^{\alpha,n}\rangle-\langle \Bar{u}^n-v^n,dAu^{\alpha,n}\rangle-\langle v^n,dAu^{\alpha,n}\rangle\\ &= -\langle (-A)^{1/2} u^{\alpha,n},d(-A)^{1/2}u^{\alpha,n}\rangle+\langle (-A)^{1/2}(\Bar{u}^n-v^n),d(-A)^{1/2}u^{\alpha,n}\rangle-\langle v^n,dAu^{\alpha,n}\rangle\\ & =-\frac{d\lVert (-A)^{1/2}u^{\alpha,n}\rVert^2}{2}+\frac{d\langle\langle (-A)^{1/2}u^{\alpha,n},(-A)^{1/2}u^{\alpha,n} \rangle\rangle_t}{2}+d\langle (-A)^{1/2}(\Bar{u}^n-v^n),(-A)^{1/2}u^{\alpha,n}\rangle\\ &- \langle (-A)^{1/2}\partial_t(\Bar{u}^n-v^n),(-A)^{1/2}u^{\alpha,n}\rangle-d\langle v^n,Au^{\alpha,n}\rangle+\langle \partial _t v^n,Au^{\alpha,n}\rangle.
\end{align*}
Therefore, we arrive to this final expression
\begin{align*}
    \lVert W^{\alpha,n}(t)\rVert^2& =\lVert W^{\alpha,n}_0\rVert^2-\alpha^2\lVert \nabla u^{\alpha,n}(t)\rVert_{L^2}^2+\alpha^2\lVert \nabla u^{\alpha,n}_0\rVert_{L^2}^2\\ &+\Tilde{\nu}\alpha^2\sum_{k\in K}\int_0^t\sum_{i=1}^n\langle G^k(u^{\alpha}(s)),\tilde{e}_i\rangle^2\lVert (-A)^{1/2}(I-\alpha^2A)^{-1}\tilde{e}_i\rVert^2ds\\ &+2\alpha^2\langle \nabla (\Bar{u}^n-v^n)(t),\nabla u^{\alpha,n}(t)\rangle_{L^2}-2\alpha^2\langle \nabla (\Bar{u}^n-v^n)_0,\nabla u^{\alpha,n}_0\rangle_{L^2}\\ & -2\alpha^2\int_0^t \langle\nabla \partial_s(\Bar{u}^n(s)-v^n(s)),\nabla u^{\alpha,n}(s)\rangle_{L^2} ds\\ &-2\alpha^2 \langle v^n(t),\Delta u^{\alpha,n}(t)\rangle_{L^2}+2\alpha^2\langle v^n_0,\Delta u^{\alpha,n}_0\rangle_{L^2}\\ &+2\alpha^2\int_0^t \langle \partial_s v^n(s),\Delta u^{\alpha,n}(s)\rangle_{L^2} ds +2\nu\int_0^t\langle W^{\alpha,n}(s),Au^{\alpha,n}(s)\rangle ds\\ &-2\int_0^t\langle W^{\alpha,n}(s),\sum_{i=1}^n b(W^{\alpha}(s),\Bar{u}(s),\tilde{e}_i)\tilde{e}_i\rangle ds-2\int_0^t\langle W^{\alpha,n}(s),\sum_{i=1}^n b(u^{\alpha}(s),W^{\alpha}(s),\tilde{e}_i)\tilde{e}_i\rangle ds\\ & -2\alpha^2\int_0^t\langle W^{\alpha,n}(s),\sum_{i=1}^n b(\tilde{e}_i,\Delta u^{\alpha}(s),u^{\alpha}(s))\tilde{e}_i\rangle ds+2\alpha^2\int_0^t\langle W^{\alpha,n}(s),\sum_{i=1}^n b(u^{\alpha}(s),\Delta u^{\alpha}(s),\tilde{e}_i)\tilde{e}_i\rangle ds\\ &+2\Tilde{\nu}\sum_{i=1}^n\int_0^t\langle F(u^{\alpha}(s)),\tilde{e}_i\rangle \langle W^{\alpha,n}(s),\tilde{e}_i\rangle ds+2\sqrt{\Tilde{\nu}}\sum_{k\in K}\int_0^t\sum_{i=1}^n \langle G^k(u^{\alpha}(s)),\tilde{e}_i\rangle \langle W^{\alpha,n}(s),\tilde{e}_i\rangle d W^k_s\\ &+\alpha^4 \Tilde{\nu} \sum_{k\in K}\int_0^t\sum_{i=1}^n \langle G^k(u^{\alpha}(s)), \tilde{e}_i\rangle^2 \lVert A(I-\alpha^2 A)^{-1}\tilde{e}_i\rVert^2  ds\\ &+\Tilde{\nu} \sum_{k\in K}\int_0^t\sum_{i=1}^n\langle G^k(u^{\alpha}(s)), \tilde{e}_i\rangle^2 ds +\alpha^2\Tilde{\nu}\sum_{k\in K}\int_0^t\sum_{i=1}^n \langle G^k(u^{\alpha}(s)), \tilde{e}_i\rangle^2\langle A(I-\alpha^2 A)^{-1}\tilde{e}_i,\tilde{e}_i\rangle ds.
\end{align*}
Now, letting $n\rightarrow +\infty$, exploiting the regularity of  $u^{\alpha},\ \Bar{u},\ v$ and the continuity of the trilinear form $b$ we arrive to the formula below
\begin{align*}
    \lVert W^{\alpha}(t)\rVert^2+\alpha^2\lVert \nabla u^{\alpha}(t)\rVert_{L^2}^2& =\lVert W^{\alpha}_0\rVert^2+\alpha^2\lVert \nabla u^{\alpha}_0\rVert_{L^2}^2\\ &+\alpha^2\Tilde{\nu}\sum_{k\in K}\int_0^t\lVert (-A)^{1/2}(I-\alpha^2A)^{-1}G^k(u^{\alpha}(s))\rVert^2ds\\ &+2\alpha^2\langle \nabla (\Bar{u}-v)(t),\nabla u^{\alpha}(t)\rangle_{L^2}-2\alpha^2\langle \nabla (\Bar{u}-v)_0,\nabla u^{\alpha}_0\rangle_{L^2}\\ & -2\alpha^2\int_0^t \langle\nabla \partial_s(\Bar{u}(s)-v(s)),\nabla u^{\alpha}(s)\rangle_{L^2} ds\\ &-2\alpha^2 \langle v(t),\Delta u^{\alpha}(t)\rangle_{L^2}+2\alpha^2\langle v_0,\Delta u^{\alpha}_0\rangle_{L^2}\\ &+2\alpha^2\int_0^t \langle \partial_s v(s),\Delta u^{\alpha}(s)\rangle_{L^2} ds +2\nu\int_0^t\langle W^{\alpha}(s),Au^{\alpha}(s)\rangle ds\\ &-2\int_0^t b(W^{\alpha}(s),\Bar{u}(s),W^{\alpha}(s))\rangle ds\\ & -2\alpha^2\int_0^t b(W^{\alpha}(s),\Delta u^{\alpha}(s),u^{\alpha}(s)) ds+2\alpha^2\int_0^tb(u^{\alpha}(s),\Delta u^{\alpha}(s),W^{\alpha}(s)) ds\\ &+2\Tilde{\nu}\int_0^t\langle F(u^{\alpha}(s)),W^{\alpha}(s)\rangle ds +2\sqrt{\Tilde{\nu}}\sum_{k\in K}\int_0^t \langle G^k(u^{\alpha}(s)),W^{\alpha}(s)\rangle  dW^k_s\\ &+\alpha^4 \Tilde{\nu} \sum_{k\in K} \int_0^t\lVert A(I-\alpha^2 A)^{-1}G^k(u^{\alpha}(s))\rVert^2 ds \\ &+\Tilde{\nu} \sum_{k\in K}\int_0^t\lVert G^k(u^{\alpha}(s))\rVert^2 ds +\alpha^2\Tilde{\nu}\sum_{k\in K}\int_0^t\langle A(I-\alpha^2A)^{-1}G^k(u^{\alpha}(s)),G^k(u^{\alpha}(s))\rangle ds\\ &= I_1(t)+I_2(t)+I_3(t)+I_4(t)+I_5(t)+I_6(t)+M(t),
\end{align*}
where:
\begin{align*}
    I_1(t)&=\lVert W^{\alpha}_0\rVert^2+\alpha^2\lVert \nabla u^{\alpha}_0\rVert_{L^2}^2+2\alpha^2\langle \nabla (\Bar{u}-v)(t),\nabla u^{\alpha}(t)\rangle_{L^2}-2\alpha^2\langle \nabla (\Bar{u}-v)_0,\nabla u^{\alpha}_0\rangle_{L^2}
    \\ &-2\alpha^2 \langle v(t),\Delta u^{\alpha}(t)\rangle_{L^2}+2\alpha^2\langle v_0,\Delta u^{\alpha}_0\rangle_{L^2},
\end{align*}
\begin{align*}
    I_2(t)&= \alpha^2\Tilde{\nu}\sum_{k\in K}\int_0^t\lVert (-A)^{1/2}(I-\alpha^2A)^{-1}G^k(u^{\alpha})\rVert^2ds+\alpha^4 \Tilde{\nu} \sum_{k\in K} \int_0^t\lVert A(I-\alpha^2 A)^{-1}G^k(u^{\alpha}(s))\rVert^2 ds\\ &+\Tilde{\nu} \sum_{k\in K}\int_0^t\lVert G^k(u^{\alpha}(s))\rVert^2 ds +\alpha^2\Tilde{\nu}\sum_{k\in K}\int_0^t\langle A(I-\alpha^2A)^{-1}G^k(u^{\alpha}(s)),G^k(u^{\alpha}(s))\rangle ds\\ & +2\Tilde{\nu}\int_0^t \langle F(u^{\alpha}(s)),W^{\alpha}(s)\rangle ds,
\end{align*}
\begin{align*}
    I_3(t) &=-2\alpha^2\int_0^t \langle\nabla \partial_s(\Bar{u}(s)-v(s)),\nabla u^{\alpha}(s)\rangle_{L^2} ds
    +2\alpha^2\int_0^t \langle \partial_s v(s),\Delta u^{\alpha}(s)\rangle_{L^2} ds,
\end{align*}
\begin{align*}
    I_4(t)&=2\nu\int_0^t\langle W^{\alpha}(s),Au^{\alpha}(s)\rangle ds,
\end{align*}
\begin{align*}
    I_5(t) &=-2\int_0^t b(W^{\alpha}(s),\Bar{u}(s),W^{\alpha}(s))\rangle ds,
\end{align*}
\begin{align*}
    I_6(t) &=-2\alpha^2\int_0^t b(W^{\alpha}(s),\Delta u^{\alpha}(s),u^{\alpha}(s)) ds+2\alpha^2\int_0^tb(u^{\alpha}(s),\Delta u^{\alpha}(s),W^{\alpha}(s)) ds,
\end{align*}
\begin{align*}
    M(t)=2\sqrt{\Tilde{\nu}}\sum_{k\in K}\int_0^t \langle G^k(u^{\alpha}(s)),W^{\alpha}(s)\rangle  dW^k_s.
\end{align*}

Our approach is almost completely pathwise. Therefore we need to estimate the terms $I_i(t),\ t\in \{1,\dots,6\}$.
The analysis of $I_1(t)$ follows by Young's inequality, the estimates on the boundary layer corrector \eqref{property bl corrector} and the interpolation estimate \eqref{interpolation estimate}

\begin{align}\label{I_1 estimate}
    I_1(t)  & \leq \lVert W_0^{\alpha}\rVert^2+\alpha^2\lVert \nabla u_0^{\alpha}\rVert_{L^2}^2+C\alpha^2\delta^{1/2}\lVert \nabla u_0^{\alpha}\rVert_{L^2}^{1/2}\lVert u_0^{\alpha}\rVert_{H^3}^{1/2} \notag\\ & + C\alpha^2(1+\delta^{-1/2})\lVert \nabla u^{\alpha}(t)\rVert_{L^2}+C\alpha^2\delta^{1/2}\lVert \nabla u(t)^{\alpha}\rVert^{1/2}_{L^2}\lVert u(t)^{\alpha}\rVert_{H^3}^{1/2} \notag\\ & +C\alpha^2(1+\delta^{-1/2})\lVert \nabla u_0^{\alpha}\rVert_{L^2}\notag\\ & \leq \lVert W_0^{\alpha}\rVert^2+C\alpha^2\lVert \nabla u_0^{\alpha}\rVert_{L^2}^2+C\alpha^2(1+\delta^{-1})+C\alpha^6\delta(\lVert u_0^{\alpha}\rVert_{H^3}^2+\lVert u(t)^{\alpha}\rVert_{H^3}^2)+C\delta^{1/2}+\frac{\alpha^2}{2}\lVert \nabla u^{\alpha}(t)\rVert_{L^2}^2.
\end{align}
The analysis of $I_2(t)$ follows by Young's inequality and the results of Lemma \ref{behavior G^k, F}, Corollary \ref{regularity of G^k, F}. Indeed it holds \begin{align}\label{I_2 estimate}
    I_2(t)&\leq C \Tilde{\nu}\sum_{k\in K} \lVert \sigma_k\rVert_{L^{\infty}}^2\int_0^t  \lVert \nabla u^{\alpha}(s)\rVert_{L^2}^2 ds +\frac{\Tilde{\nu}}{\alpha}\sum_{k\in K} \lVert \sigma_k\rVert_{L^{\infty}}^2 \int_0^t \lVert  W^{\alpha}(s)\rVert \lVert \nabla u^{\alpha}(s)\rVert_{L^2} ds \notag \\ & \leq 
    C \Tilde{\nu}\sum_{k\in K} \lVert \sigma_k\rVert_{L^{\infty}}^2\int_0^t  \lVert \nabla u^{\alpha}(s)\rVert_{L^2}^2 ds+\sum_{k\in K} \lVert \sigma_k\rVert_{L^{\infty}}^2 \int_0^t \lVert  W^{\alpha}(s)\rVert^2 ds+\left(\frac{\Tilde{\nu}}{\alpha}\right)^2\sum_{k\in K} \lVert \sigma_k\rVert_{L^{\infty}}^2 \int_0^t  \lVert \nabla u^{\alpha}(s)\rVert_{L^2}^2 ds.
\end{align}

The analysis of $I_3(t)$ follows by Young's inequality, the estimates on the boundary layer corrector \eqref{property bl corrector} and the interpolation estimate \eqref{interpolation estimate}
\begin{align}\label{I_3 estimate }
    I_3(t)&\leq C\alpha^2(1+\delta^{-1/2})\int_0^t\lVert \nabla u^{\alpha}(s)\rVert_{L^2} ds+C\alpha^2\delta^{1/2}\int_0^t \lVert \nabla u^{\alpha}(s)\rVert_{L^2}^{1/2}\lVert u^{\alpha}(s)\rVert_{H^3}^{1/2} \notag\\ & \leq C\delta^{1/2}+C\alpha^2(1+\delta^{-1})+C\alpha^2\int_0^t \lVert \nabla u^{\alpha}(s)\rVert_{L^2}^2 ds +C\alpha^6\delta\int_0^t\lVert u^{\alpha}(s)\rVert_{H^3}^2 ds.
\end{align}
The analysis of $I_4(t)$ is analogous to equations (3.20)-(3.21)-(3.22) in \cite{lopes2015approximation}, it implies: \begin{align*}
        2\nu\int_0^t\langle W^{\alpha}(s),Au^{\alpha}(s)\rangle ds &\leq -2\nu \int_0^t\lVert \nabla u^{\alpha}(s)\rVert_{L^2}^2 ds+C\frac{\nu}{\alpha}(1+\delta^{-1/2})\int_0^t\alpha \lVert\nabla u^{\alpha}(s)\rVert_{L^2} ds\\ & +\frac{C\nu \delta^{1/2}}{\alpha^2}\int_0^t \alpha^2\lVert \Delta u^{\alpha}(s)\rVert_{L^2} ds.
    \end{align*}
     Therefore by the interpolation inequality \eqref{interpolation estimate} and Young's inequality we have \begin{align}\label{I_4 estimate}
     2\nu\int_0^t\langle W^{\alpha}(s),Au^{\alpha}(s)\rangle ds &\leq -2\nu \int_0^t \lVert \nabla u^{\alpha}(s)\rVert_{L^2}^2 ds+C\alpha^2 \int_0^t \lVert \nabla u^{\alpha}(s)\rVert_{L^2}^2 ds \notag\\ & +C\alpha^6\delta\int_0^t\lVert u^{\alpha}(s)\rVert_{H^3}^2 ds+C\left(\frac{\nu}{\alpha^2}\right)^2\delta^{1/2}+C\left(\frac{\nu}{\alpha}\right)^2(1+\delta^{-1}).
 \end{align}
 The analysis of $I_5(t)$ follows immediately by H\"older's inequality:
 \begin{align}\label{I_5 estimate}
I_5(t)\leq \lVert \bar{u}\rVert_{L^{\infty}(0,T;H^3)}\int_0^t\lVert W^{\alpha}(s)\rVert^2 ds.
 \end{align}
For what concerns the analysis of $I_6(t)$, preliminary we observe that
\begin{align*}
       -2\alpha^2 b(W^{\alpha},\Delta u^{\alpha},u^{\alpha}) +2\alpha^2 b(u^{\alpha},\Delta u^{\alpha},W^{\alpha})&=2\alpha^2 b(\Bar{u},\Delta u^{\alpha},u^{\alpha})-2\alpha^{2}b(u^{\alpha},\Delta u^{\alpha},u^{\alpha})\\ & +2\alpha^2 b(u^{\alpha}, \Delta u^{\alpha}, u^{\alpha})-2\alpha^{2}b(u^{\alpha},\Delta u^{\alpha},\Bar{u}).
    \end{align*}
    Arguing as in \cite{lopes2015convergence}, equations (4.18)-(4.19) we get
    \begin{align}\label{I_6 estimate}
        I_6(t)&\leq C \alpha^2\left(1+\lVert \Bar{u}\rVert_{L^{\infty}(0,T;H^3)}\right)\int_0^t \lVert \nabla u^{\alpha}(s)\rVert_{L^2}^2ds+C\alpha^2\lVert \Bar{u}\rVert_{L^{\infty}(0,T;H^3)}^4\int_0^t \lVert  u^{\alpha}(s)\rVert^2 ds.
    \end{align}
Combining equations \eqref{I_1 estimate},\eqref{I_2 estimate},\eqref{I_3 estimate },\eqref{I_4 estimate},\eqref{I_5 estimate},\eqref{I_6 estimate} and exploiting our assumptions on the behavior of $\nu,\ \Tilde{\nu}, \alpha^2$, see Hypothesis \ref{Hypothesis inviscid limit}, we have the integral relation below: \begin{align}\label{pathwise estimate}
    \lVert W^{\alpha}(t)\rVert^2+\frac{\alpha^2}{2}\lVert \nabla u^{\alpha}(t)\rVert_{L^2}^2 &\leq M(t)+C\alpha^2(1+\delta^{-1})+C\delta^{1/2}+\lVert W^{\alpha}_0\rVert^2+C\alpha^2\lVert \nabla u_0^{\alpha}\rVert_{L^2}^2 \notag+C\alpha^6\delta(\lVert u_0^\alpha\rVert_{H^3}^2+\lVert u(t)^\alpha\rVert_{H^3}^2) \\ & +C_{\{\sigma_k\}_{k\in K}}\int_0^t \lVert W^{\alpha}(s)\rVert^2+{\alpha^2}\lVert \nabla u^{\alpha}(s)\rVert_{L^2}^2 ds+C\alpha^6\delta\int_0^t\lVert  u^{\alpha}(s)\rVert_{H^3}^2 ds+C\alpha^2\int_0^t\lVert u^{\alpha}(s)\rVert^2 ds.
\end{align}
By the stochastic Gr\"onwall's Lemma \ref{stochastic gronwall } above we have:
\begin{align}
    &\operatorname{sup}_{t\in [0,T]}\mathbb{E}\left[\lVert W^{\alpha}(t)\rVert^2\right]+\alpha^2\operatorname{sup}_{t\in [0,T]}\mathbb{E}\left[\lVert \nabla u^{\alpha}(t)\rVert_{L^2}^2\right]\notag\\ &\leq C_{\{\sigma_k\}_{k\in K}}\left(\alpha^2\mathbb{E}\left[\int_0^T \lVert u^{\alpha}(s)\rVert^2 ds\right]+\alpha^6\delta \int_0^T \lVert u^{\alpha}(s)\rVert_{H^3}^2 ds+\alpha^6\delta\mathbb{E}\left[\operatorname{sup}_{t\in[0,T]}\lVert u^{\alpha}(t)\rVert_{H^3}^2\right]\right)\notag\\ & +C_{\{\sigma_k\}_{k\in K}}\left(\alpha^2(1+\delta^{-1})+\delta^{1/2}+\mathbb{E}\left[\lVert W^{\alpha}_0\rVert^2+\alpha^2\lVert \nabla u_0^{\alpha}\rVert_{L^2}^2 +\alpha^6\delta\lVert u_0^\alpha\rVert_{H^3}^2\right]\right).
\end{align}
Thanks to Hypothesis \ref{Hypothesis inviscid limit} and our assumptions on $\delta$, see equation \eqref{assumptions delta}, we have that
\begin{align}\label{prefinal 1}
    \alpha^2(1+\delta^{-1})+\delta^{1/2}+\mathbb{E}\left[\lVert W^{\alpha}_0\rVert^2+\alpha^2\lVert \nabla u_0^{\alpha}\rVert_{L^2}^2 +\alpha^6\delta\lVert u_0^\alpha\rVert_{H^3}^2\right]\rightarrow 0.
\end{align}
Thanks to Lemma \ref{thm energy inequality}, we have that
\begin{align}\label{prefinal 2}
\alpha^2\mathbb{E}\left[\int_0^T \lVert u^{\alpha}(s)\rVert^2 ds\right]+\alpha^6\delta \int_0^T \lVert u^{\alpha}(s)\rVert_{H^3}^2 ds+\alpha^6\delta\mathbb{E}\left[\operatorname{sup}_{t\in[0,T]}\lVert u^{\alpha}(t)\rVert_{H^3}^2\right]\rightarrow 0.    
\end{align}
Therefore \begin{align}\label{weaker convergence final}
    \operatorname{sup}_{t\in [0,T]}\mathbb{E}\left[\lVert W^{\alpha}(t)\rVert^2\right]+\alpha^2\operatorname{sup}_{t\in [0,T]}\mathbb{E}\left[\lVert \nabla u^{\alpha}(t)\rVert_{L^2}^2\right]\rightarrow 0.
\end{align}
Restarting from equation \eqref{pathwise estimate} and considering the expected value of the supremum of both the terms in the left hand side we have
\begin{align}\label{prefinal strong convergence}
    &\mathbb{E}\left[\operatorname{sup}_{t\in [0,T]}\lVert W^{\alpha}(t)\rVert^2\right]+\alpha^2\mathbb{E}\left[\operatorname{sup}_{t\in [0,T]}\lVert \nabla u^{\alpha}(t)\rVert_{L^2}^2\right]\notag \\ & \leq C\left(\alpha^2\mathbb{E}\left[\int_0^T \lVert u^{\alpha}(s)\rVert^2 ds\right]+\alpha^6\delta \int_0^T \lVert u^{\alpha}(s)\rVert_{H^3}^2 ds+\alpha^6\delta\mathbb{E}\left[\operatorname{sup}_{t\in[0,T]}\lVert u^{\alpha}(t)\rVert_{H^3}^2\right]\right)\notag\\ & +C\left(\alpha^2(1+\delta^{-1})+\delta^{1/2}+\mathbb{E}\left[\lVert W^{\alpha}_0\rVert^2+\alpha^2\lVert \nabla u_0^{\alpha}\rVert_{L^2}^2 +\alpha^6\delta\lVert u_0^\alpha\rVert_{H^3}^2\right]\right)\notag\\ & +C\mathbb{E}\left[\operatorname{sup}_{t\in [0,T]}M(t)\right]+C_{\{\sigma_k\}_{k\in K}}\mathbb{E}\left[\int_0^T \lVert W^{\alpha}(s)\rVert^2+{\alpha^2}\lVert \nabla u^{\alpha}(s)\rVert_{L^2}^2 ds\right].
\end{align} We already proved that almost all the terms in the right hand side of equation \eqref{prefinal strong convergence} go to $0$.
Therefore in order to complete the proof we left to show that \begin{align*}
    \mathbb{E}\left[\operatorname{sup}_{t\in [0,T]}M(t)\right]+\mathbb{E}\left[\int_0^t \lVert W^{\alpha}(s)\rVert^2+{\alpha^2}\lVert \nabla u^{\alpha}(s)\rVert_{L^2}^2 ds\right]\rightarrow 0.
\end{align*}By the weaker convergence described by equation \eqref{weaker convergence final} and Fubini Theorem \begin{align*}
    \mathbb{E}\left[\int_0^t \lVert W^{\alpha}(s)\rVert^2+{\alpha^2}\lVert \nabla u^{\alpha}(s)\rVert_{L^2}^2 ds\right]\rightarrow 0.
\end{align*} For what concerns the other, the convergence follows by Burkholder-Davis-Gundy inequality, Hypothesys \ref{Hypothesis inviscid limit}, equation \eqref{weaker convergence final}, Fubini Theorem and Lemma \eqref{energy estimates}. Indeed \begin{align*}
    \mathbb{E}\left[\operatorname{sup}_{t\in [0,T]}M(t)\right]& \leq C   \sqrt{\Tilde{\nu}}\mathbb{E}\left[\left(\sum_{k\in K}\int_0^T\lVert G^k(u^{\alpha}(s))\rVert^2 \lVert W^{\alpha}(s)\rVert^2ds\right)^{1/2}\right]\\ & 
        \leq C\left(\sum_{k\in K}\lVert \sigma_k\rVert_{L^{\infty}}^2\right)^{1/2} \sqrt{\Tilde{\nu}}\mathbb{E}\left[\left(\int_0^T\lVert \nabla u^{\alpha}(s)\rVert_{L^2}^2\lVert W^{\alpha}(s)\rVert^2ds\right)^{1/2}\right]\\ & \leq C\left(\sum_{k\in K}\lVert \sigma_k\rVert_{L^{\infty}}^2\right)^{1/2} \sqrt{\Tilde{\nu}}\mathbb{E}\left[\operatorname{sup}_{t\in [0,T]}\lVert \nabla u^{\alpha}(t)\rVert_{L^2}\left(\int_0^T \lVert W^{\alpha}(s)\rVert^2ds\right)^{1/2}\right]\\ & \leq C\left(\sum_{k\in K}\lVert \sigma_k\rVert_{L^{\infty}}^2\right)^{1/2} \mathbb{E}\left[\int_0^T \lVert W^{\alpha}(s)\rVert^2ds\right]^{1/2}\left(\mathbb{E}\left[\alpha^2\operatorname{sup}_{t\in [0,T]}\lVert \nabla u^{\alpha}(t)\rVert_{L^2}^2\right]\right)^{1/2}\\ &\rightarrow 0.
\end{align*}
Now the proof is complete.
\end{proof}

\begin{remark}
Combining Lemma \ref{thm energy inequality} and Theorem \ref{Main thm} we understand that, if $\nu=O(\alpha^2)$ and $\Tilde{\nu}=O(\alpha^2)$, the assumptions on the behavior of the initial conditions $u_0^{\alpha}$ in norm $H,\ H^1$ and $H^3$ are satisfied also for $t\in [0,T]$.
\end{remark}

\section{The Case of Additive Noise}\label{additive noise}
For what concerns the case with additive noise, as stated in Section \ref{results}, the well-posedness is a well-known fact in case of $\nu>0$ and we can prove a result completely analogous to Theorem \ref{Main thm}, following exactly the same argument. However, the restriction $\nu>0$ can be omitted modifying slightly the proof of \cite{razafimandimby2012strong} as described in Remarks \ref{changes Lemma energy W}, \ref{changes crucial lemma} and \ref{remark thm energy inequality}. However, we do not stress this assumption in this Section, therefore $\nu>0$ in what follows. What was crucial for the proof of Theorem \ref{Main thm} were the energy estimates of Lemma \ref{thm energy inequality}. Thus in this Section we want to explain a different approach to prove these energy estimates in the case of additive noise. These computations are more similar to what happens in the deterministic framework. We keep previous assumptions on the coefficients $\sigma_k$ and the Brownian motions $W^k$. For generality reasons we consider the equations without any scaling factor on the noise. Thus we consider
\begin{align}
\begin{cases}\label{system}
d v=(\nu \Delta u- \operatorname{curl}(v)\times u+\nabla p)dt +\sum_{k\in K} \sigma_k dW^k_t\\ 
\operatorname{div}u=0\\
v=u-\alpha^2\Delta u\\
u|_{\partial D}=0\\ 
u(0)=u_0
\end{cases}
\end{align}
Before going on, we need to recall a result of \cite{lopes2015convergence}.
\begin{Lemma}\label{well posed elliptic}
Let $q\in L^2(D)$, there exists a unique $\phi\in H^2_0(D)$ solution of \begin{align*}
    \begin{cases}
    \Delta \phi -\alpha^2\Delta^2\phi=q\\ 
    \phi|_{\partial D}=\partial_n\phi|_{\partial D}=0
    \end{cases}
\end{align*}
which satisfies
$$\langle \nabla \phi ,\nabla v\rangle_{L^2}+\alpha^2 \langle \Delta \phi, \Delta v\rangle_{L^2}=-\langle q,v\rangle \text{ for each }v\in H^2_0.  $$
Moreover, the solution map is continuous from $L^2(D)$ to $H^2_0(D)\cap H^4(D).$
\end{Lemma}
Thanks to this Lemma, we can define an operator $\mathbb{K}:L^2(D)\rightarrow H^3(D)\cap W^{1,\infty}_0(D)$ which associates to each $q\in L^2(D)$ the vector field $u=\nabla^{\perp} \phi$, where $\phi$ is the solution of the equation of Lemma \ref{well posed elliptic}.

\begin{definition}\label{weak Solution system} 
A stochastic process $u$ weakly continuous with values in $W$ and continuous with values in $V$
is a weak solution of equation (\ref{system}) if
$$
u \in L^p(\Omega,\mathcal{F},\mathbb{P};L^{\infty}(0,T;W)),\ \forall p\geq 2.
$$
and $\mathbb{P}-a.s.$ for every $t\in [0,T]$ and $\phi \in D(A)$  we have
\begin{align*}
    \langle u(t), (I-\alpha^2 A)\phi\rangle-  \langle u_0, (I-\alpha^2 A)\phi\rangle&=\nu \int_0^t \langle u(s), A\phi\rangle ds- \int_0^t b(u(s), u(s)-\alpha^2 \Delta u(s),\phi)ds\\ &-\alpha^2 \int_0^t b(\phi, \Delta u(s),u(s))ds+\sum_{k\in K} \langle \sigma_k,\phi\rangle W^k_t.
\end{align*}
\end{definition}
Arguing as in the first part of the proof of Lemma \ref{thm energy inequality} we can prove the following result.
\begin{Lemma}\label{Ito system}
Let $u$ be a weak solution of problem (\ref{system}) in the sense of Definition \ref{weak Solution system}, then the following relations hold true
\begin{enumerate}
    \item $$ d\lVert u\rVert^2+\alpha^2d\lVert \nabla u\rVert_{L^2}^2=(-2\nu\lVert \nabla u\rVert_{L^2}^2+\sum_{k\in K} \langle \sigma_k,(I-\alpha^2A)^{-1}\sigma_k\rangle)dt+ 2\sum_{k\in K} \langle \sigma_k,u \rangle dW^k_t$$
 \item \begin{align*}
     &\mathbb{E}\left[\lVert u(t)\rVert^2\right]+\alpha^2\mathbb{E}\left[\lVert \nabla u(t)\rVert_{L^2}^2\right]+2\nu\int_0^t \mathbb{E}\left[\lVert \nabla u(s)\rVert_{L^2}^2\right]ds=\\& \mathbb{E}\left[\lVert u_0\rVert^2\right]+\alpha^2\mathbb{E}\left[\lVert \nabla u_0\rVert_{L^2}^2\right]+t\sum_{k\in K} \langle \sigma_k,(I-\alpha^2A)^{-1}\sigma_k\rangle
 \end{align*} 
 \item \begin{align*}
     &\mathbb{E}\left[\operatorname{sup}_{t\in [0,T]}\lVert u(t)\rVert^2\right]+\alpha^2\mathbb{E}\left[\operatorname{sup}_{t\in [0,T]}\lVert \nabla u(t)\rVert_{L^2}^2\right]+2\nu\int_0^T \mathbb{E}\left[\lVert \nabla u(s)\rVert_{L^2}^2\right]ds\leq \\& C\left( \mathbb{E}\left[\lVert u_0\rVert^2\right]+\alpha^2\mathbb{E}\left[\lVert \nabla u_0\rVert_{L^2}^2\right]+T\sum_{k\in K} \langle \sigma_k,(I-\alpha^2A)^{-1}\sigma_k\rangle+\mathbb{E}\left[\left(\int_0^T \sum_{k\in K}\langle \sigma_k,u(s) \rangle^2ds\right)^{1/2}\right]\right)
\end{align*}
\end{enumerate}
\end{Lemma}
Let us introduce the vorticity formulation of (\ref{system}), we denote $s_k=\operatorname{curl}\sigma_k$
\begin{align}\label{vorticity}
\begin{cases}
     dq+\left(\frac{\nu}{\alpha^2}(q-\operatorname{curl}u)+u\cdot\nabla q\right) dt=\sum_{k\in K} s_k  dW^k_t \\ 
    \operatorname{div}u=0\\ 
    q=\operatorname{curl}(u-\alpha^2\Delta u)\\ 
    q(0)=q_0:=\operatorname{curl}(u_0-\alpha^2\Delta u_0)\\
    u|_{\partial D}=0    
\end{cases}
\end{align}
\begin{definition}\label{well posed vorticity non linear}
A stochastic process $q$, which is weakly continuous with values in $L^2(D)$ and continuous with values in $H^{-1}(D)$,
is a weak solution of equation (\ref{vorticity}) if
$$
q \in L^p(\Omega,\mathcal{F},\mathbb{P};L^{\infty}(0,T;L^2(D))),\ \forall p\geq 2.
$$
and $\mathbb{P}-a.s.$ for every $t\in [0,T]$ and $\phi \in H^2_0(D)$  we have
 \begin{align*}
    \langle q(t),\phi\rangle-\langle q_0,\phi\rangle& =\int_0^t\int_D u(s)\cdot \nabla \phi q(s) \ dx ds-\frac{\nu}{\alpha^2}\int_0^t\int_D (q(s)-\operatorname{curl}u(s))\phi\ dx ds\\ & +\sum_{k\in K} \langle s_k, \phi \rangle W^k_t\ \ \mathbb{P}-a.s.
\end{align*} $u=\nabla^{\perp}\varphi,$ $\varphi$ obtained by Lemma \ref{well posed elliptic}, $ u\in W.$
\end{definition}

Let us obtain a result about the equivalence between the solutions of these two problems. Since we know from the results of \cite{razafimandimby2012strong} that problem (\ref{system}) is well-posed, then problem (\ref{vorticity}) is well-posed as well.
\begin{proposition}\label{equivalence}Let $u$ be a solution of (\ref{system}) in the sense of Definition \ref{weak Solution system}, then $q:=\operatorname{curl}(u-\alpha^2\Delta u)$ is a solution of (\ref{vorticity}) in the sense of Definition \ref{well posed vorticity non linear}. Conversely, if $q$ is a solution of (\ref{vorticity}) in the sense of Definition \ref{well posed vorticity non linear} then $u=\nabla^{\perp}\varphi$, $\varphi$ obtained by Lemma \ref{well posed elliptic}, is a solution of (\ref{system}) in the sense of Definition \ref{weak Solution system}.
\end{proposition}
\begin{proof}
\begin{itemize}
    \item$Def.\ref{weak Solution system}\implies Def.\ref{well posed vorticity non linear}$ is immediate taking as test function for problem (\ref{system}) $\phi=-\nabla^{\perp} \Tilde{\phi}$, $\Tilde{\phi}\in H^2_0(D)$.
    \item $Def.\ref{well posed vorticity non linear}\implies Def.\ref{weak Solution system}$ we take $u=\nabla^{\perp}\varphi,\ v=u-\alpha^{2}\Delta u,$ where  $\varphi$ is obtained by Lemma \ref{well posed elliptic} and $\phi=-\nabla^{\perp}\Tilde{\phi},$ where $\Tilde{\phi}\in H^2_0(D)$. Then integrating by parts and exploiting that $\operatorname{curl}\nabla^{\perp}=\Delta$, $\Delta \varphi -\alpha^2\Delta^2\varphi=q$ and $q$ is a solution of (\ref{vorticity}) in the sense of Definition \ref{well posed vorticity non linear} we get
    \begin{align*}
        &-\langle (I-\alpha^2\Delta)u(t),\nabla^{\perp}\Tilde{\phi}\rangle_{L^2}+\langle (I-\alpha^2\Delta)u_0,\nabla^{\perp}\Tilde{\phi}\rangle_{L^2}-\frac{\nu}{\alpha^2}\int_0^t \langle v(s)-u(s),\nabla^{\perp}\Tilde{\phi}\rangle ds\\ &+\int_0^t\int_D (u(s)\cdot \nabla) \nabla^{\perp}\Tilde{\phi}v(s)\ dx ds-\alpha^2 \int_0^t \int_D (\nabla^{\perp}\Tilde{\phi}\cdot\nabla) \Delta u(s)u(s) dxds+\sum_{k\in K}\langle \sigma_k,\nabla^{\perp}{\Tilde{\phi}}\rangle W^k_t=0 \ \ \mathbb{P}-a.s.
    \end{align*} From the last relation the thesis follows if we are able to prove the continuity properties of $u$. The weak continuity of $u$ with values in $W$ follows immediately from the regularity of $q$ and Lemma \ref{well posed elliptic}. Again by Lemma \ref{well posed elliptic} we get the strong continuity of $u$ with values in $V$. In fact, via Lax-Milgram Lemma we get the regularity of the solution mapping of the problem described in Lemma \ref{well posed elliptic}
between $H^{-2}(D)$ and $H^2_0(D)$. Via interpolation techniques we recover the regularity of the solution mapping between $H^{-1}(D)$ and $H^3(D)\cap H^2_0(D)$, therefore the required regularity for $u$.
\end{itemize} 
\end{proof}

Approximating the process $q(t)$ solution of (\ref{vorticity}) by the eigenvectors of the Laplacian with Dirichlet boundary conditions and then arguing as in the first part of the proof of Lemma \ref{thm energy inequality}, we can obtain some It\^o's formula and energy estimates. Moreover, if $u\in V$ we have $\lVert \nabla u\rVert_{L^2}^2=\lVert \operatorname{curl}u\rVert_{L^2}^2$. Thanks to Proposition \ref{equivalence}, we know that $u$ appearing in problem (\ref{vorticity}) is a solution of problem (\ref{system}). Therefore, thanks to Lemma \ref{Ito system} we know that
\begin{align}\label{additive noise H1 bound 1}
   2\nu\int_0^t \mathbb{E}\left[\lVert \nabla u(s)\rVert_{L^2}^2\right]ds\leq  \mathbb{E}\left[\lVert u_0\rVert^2\right]+\alpha^2\mathbb{E}\left[\lVert \nabla u_0\rVert_{L^2}^2\right]+t\sum_{k\in K} \langle \sigma_k,(I-\alpha^2A)^{-1}\sigma_k\rangle, \\
   \label{additive noise H1 bound 2} \alpha^2\mathbb{E}\left[\lVert \nabla u(t)\rVert_{L^2}^2\right]\leq  \mathbb{E}\left[\lVert u_0\rVert^2\right]+\alpha^2\mathbb{E}\left[\lVert \nabla u_0\rVert_{L^2}^2\right]+t\sum_{k\in K} \langle \sigma_k,(I-\alpha^2A)^{-1}\sigma_k\rangle
\end{align}
$$  $$ and we can obtain the following energy relations. 
\begin{Lemma}\label{Ito Vorticity}
Let $q$ be a weak solution of problem (\ref{vorticity}) in the sense of Definition \ref{well posed vorticity non linear}, then the following relations hold true
\begin{enumerate}
\item \begin{align*}
    d\lVert q\rVert^2=-\frac{2\nu}{\alpha^2}\langle q-\operatorname{curl}u,q\rangle dt+\sum_{k\in K}\lVert s_k\rVert^2\ dt +2\sum_{k\in K} \langle s_k,q\rangle dW^k_t
\end{align*}
    \item \begin{align*}
     \mathbb{E}\left[\lVert q(t)\rVert^2\right]&\leq e^{-\frac{\nu}{\alpha^2}t} \mathbb{E}\left[\lVert q_0\rVert^2\right]+\frac{\alpha^2}{\nu}(1-e^{-\frac{\nu t}{\alpha^2}})\sum_{k\in K}\lVert s_k\rVert^2 \\ &+\frac{1}{2\nu}\left(\mathbb{E}\left[\lVert u_0\rVert^2\right]+\alpha^2\mathbb{E}\left[\lVert \nabla u_0\rVert_{L^2}^2\right]+T\sum_{k\in K} \langle \sigma_k,(I-\alpha^2A)^{-1}\sigma_k\rangle \right)
 \end{align*} 
    \item \begin{align*}
\mathbb{E}\left[\operatorname{sup}_{t\in [0,T]}\lVert q(t)\rVert^2\right] &\leq \mathbb{E}\left[\lVert q_0\rVert^2\right] +\sum_{k\in K} \lVert s_k\rVert^2 T\\ &+ C\mathbb{E}\left[\left(\sum_{k\in K}\int_0^T\langle s_k,q(s) \rangle^2ds\right)^{1/2}\right]\\ &+\frac{1}{2\alpha^2}\left(\mathbb{E}\left[\lVert u_0\rVert^2\right]+\alpha^2\mathbb{E}\left[\lVert \nabla u_0\rVert_{L^2}^2\right]+T\sum_{k\in K} \langle \sigma_k,(I-\alpha^2A)^{-1}\sigma_k\rangle\right).
\end{align*}
\end{enumerate}
\end{Lemma}

\begin{remark}\label{enrgy H3}
We can control the $H^3$ norm of $u$ via the $H^1$ norm of $u$ and the $L^2$ norm of $q$ in the following way \begin{align}\label{preliminary H3 additive noise}
    \lVert u(t)\rVert_{H^3}&\leq C\left( \lVert \nabla u(t)\rVert_{L^2} +\lVert \operatorname{curl}\Delta u(t)\rVert_{L^2}\right)\notag\\ &\leq C\left( \frac{\lVert q(t) \rVert}{\alpha^2}+\frac{\lVert \operatorname{curl} u(t)\rVert_{L^2}}{\alpha^2}+\lVert \nabla u(t)\rVert_{L^2}\right)\notag\\ &=C\left(\frac{\lVert q(t) \rVert}{\alpha^2}+\frac{\lVert \nabla u(t)\rVert_{L^2}}{\alpha^2}+\lVert \nabla u(t)\rVert_{L^2}\right). 
\end{align}
Therefore, thanks to Lemma \ref{Ito Vorticity} it holds

\begin{align}\label{ H3 additive 1}
     \mathbb{E}\left[\lVert u(t)\rVert_{H^3}^2\right]&\lesssim \frac{e^{-\frac{\nu}{\alpha^2}t} }{\alpha^4}\mathbb{E}\left[\lVert q_0\rVert^2\right]+\frac{1}{\nu\alpha^2}(1-e^{-\frac{\nu t}{\alpha^2}})\sum_{k\in K}\lVert s_k\rVert^2 \notag\\ &+\left(\frac{1}{\alpha^2}+\frac{1}{\alpha^6}+\frac{1}{\nu\alpha^4}\right)\left(\mathbb{E}\left[\lVert u_0\rVert^2\right]+\alpha^2\mathbb{E}\left[\lVert \nabla u_0\rVert_{L^2}^2\right]+T\sum_{k\in K} \langle \sigma_k,(I-\alpha^2A)^{-1}\sigma_k\rangle \right)
 \end{align} 
\begin{align}\label{ H3 additive 2}
\mathbb{E}\left[\operatorname{sup}_{t\in [0,T]}\lVert u(t)\rVert_{H^3}^2\right] &\lesssim \frac{1}{\alpha^4}\mathbb{E}\left[\lVert q_0\rVert^2\right] +\frac{1}{\alpha^4}\sum_{k\in K} \lVert s_k\rVert^2 T\notag\\ &+ \frac{1}{\alpha^4}\mathbb{E}\left[\left(\sum_{k\in K}\int_0^T\langle s_k,q(s) \rangle^2ds\right)^{1/2}\right]\notag\\ &+\left(\frac{1}{\alpha^6}+\frac{1}{\alpha^4}\right)\left(\mathbb{E}\left[\lVert u_0\rVert^2\right]+\alpha^2\mathbb{E}\left[\lVert \nabla u_0\rVert_{L^2}^2\right]+T\sum_{k\in K} \langle \sigma_k,(I-\alpha^2A)^{-1}\sigma_k\rangle\right)\notag\\
& +\left(\frac{1}{\alpha^6}+\frac{1}{\alpha^4}\right) \left(\sum_{k\in K}\mathbb{E}\left[\int_0^T\langle \sigma_k,u(s) \rangle^2ds\right]^{1/2} \right)
\end{align}

\end{remark}
\begin{remark}
If we consider the scaled equations with $\sqrt{\Tilde{\nu}}$ in front of the noise, then each $\sigma_k$ and $s_k$ is multiplied by $\sqrt{\Tilde{\nu}}$ in Lemmas \ref{Ito system}, \ref{Ito Vorticity} and Remark \ref{enrgy H3}.
\end{remark}
Thanks to Remark \ref{enrgy H3}, if we consider the scaled equation with additive noise and initial condition $u_0^{\alpha}$ satisfying Hypothesis \ref{Hypothesis inviscid limit}, then the following result follows immediately. 
\begin{Lemma}\label{behavior H3 H1 additive}
If we consider the stochastic second-grade fluid equations with additive noise (\ref{system}) scaled by $\sqrt{\Tilde{\nu}}$, under Hypothesis \ref{hypothesis well-posedness}-\ref{Hypothesis inviscid limit}, if $u^{\alpha}$ is the solution in the sense of Definition \ref{weak Solution system} of the problem with initial condition $u_0^{\alpha}$, then

\begin{align*}
 &\mathbb{E}\left[\operatorname{sup}_{t\in [0,T]}\lVert u(t)\rVert^2\right]+\alpha^2\mathbb{E}\left[\operatorname{sup}_{t\in [0,T]}\lVert \nabla u(t)\rVert_{L^2}^2\right]+2\nu\int_0^T \mathbb{E}\left[\lVert \nabla u(s)\rVert_{L^2}^2\right]ds=O(1),\\ 
    &\mathbb{E}\left[\alpha^6\operatorname{sup}_{t\in [0,T]}\lVert u^{\alpha}(t)\rVert_{H^3}^2\right]=O(1).
\end{align*}
\end{Lemma}

Looking carefully at the proof of Theorem \ref{Main thm}, Lemma \ref{behavior H3 H1 additive} contains the crucial bounds on the norm of the solutions to obtain the inviscid limit. Therefore, following the same ideas of Section \ref{proof of main thm}, one can prove that the inviscid limit holds:
\begin{theorem}\label{Main thm additive noise}
Under Hypotheses \ref{General Hypothesis}-\ref{hypothesis well-posedness}-\ref{Hypothesis inviscid limit}, calling $u^{\alpha}$ the solution of the stochastic second-grade fluid equations with additive noise (\ref{system}) scaled by $\sqrt{\Tilde{\nu}}$ and $\bar{u}$ the solution of (\ref{Euler equations}), then \begin{align*}
    \lim_{\alpha \rightarrow 0}\mathbb{E}\left[\sup_{t\in [0,T]}\lVert u^{\alpha}(t)-\bar{u}(t)\rVert^2\right]= 0.
\end{align*}
\end{theorem}
\section*{Acknowledgement}
I want to thank Professor Franco Flandoli and Professor Edriss Titi for useful discussions and valuable insights into the subject.
\section*{Conflict of Interest}
The author declares no conflict of interest.

  \bibliography{demo}
  \bibliographystyle{abbrv}

\end{document}